\documentclass[a4paper,12pt, reqno]{amsart}
\usepackage{amssymb}
\usepackage{amsfonts}
\usepackage{mathrsfs,amsmath,amssymb,latexsym,amsfonts, amscd} 
\usepackage[all,ps,cmtip]{xy}

\let\mathcal\mathscr

\def\lra{\longrightarrow}

\def\llra{\hbox to 10mm{\rightarrowfill}}

\def\lllra{\hbox to 15mm{\rightarrowfill}}

\def\phi{{\varphi}}

\def\cD{\mathcal{D}}

\def\cO{\mathcal{O}}

\def\cH{\mathcal{H}}

\def\cR{\mathcal{R}}
\def\cJ{\mathcal{J}}

\def\ta{\tilde{a}}

\let\tilde\widetilde

\def\LLC{\text{LLC}}
\def\Nklt{\text{Nklt}}

\DeclareMathOperator{\codim}{codim}
\DeclareMathOperator{\Pic}{Pic}

\DeclareMathOperator{\vol}{vol}

\DeclareMathOperator{\mult}{mult}

\newcommand{\fra}{\mathfrak{a}}
\newcommand{\frb}{\mathfrak{b}}

\newcommand{\bQ}{{\mathbb Q}}
\newcommand{\bP}{{\mathbb P}}
\newcommand{\roundup}[1]{\lceil{#1}\rceil}
\newcommand{\rounddown}[1]{\lfloor{#1}\rfloor}
\newcommand\lrw{\longrightarrow}
\newcommand\rw{\rightarrow}

\newcommand\OO{{\mathcal{O}}}

\newtheorem{lemm}{Lemma}[section]
\newtheorem{theo}[lemm]{Theorem}
\newtheorem{coro}[lemm]{Corollary}
\newtheorem{prop}[lemm]{Proposition}

\theoremstyle{definition}
\newtheorem{defi}[lemm]{Definition}
\newtheorem{rema}[lemm]{Remark}

\newtheorem{conj}[lemm]{Conjecture}
\newtheorem{exam}[lemm]{Example}

\newtheorem{qu}[lemm]{Question}
\newtheorem{vari}[lemm]{Variant}
\newtheorem*{thm}{Theorem 0}

\theoremstyle{remark}
\newtheorem*{remark*}{Remark}
\newtheorem*{note*}{Note}

\title[A reduction of canonical stability index]{A reduction of canonical stability index of 4 and 5 dimensional projective varieties with large volume}
\author{Meng Chen}
\address{\rm School of Mathematical Sciences \& Shanghai Centre for Mathematical Sciences, Fudan University, Shanghai 200433, China}\email{mchen@fudan.edu.cn}

\author{Zhi Jiang}
\address{\rm D\'epartement de Math\'ematiques, B\^{a}timent 425,
Universit\'{e} Paris-Sud,
F-91405 Orsay, France} \email{zhi.jiang@math.u-psud.fr}
\address{Shanghai center for mathematical sciences\\
22 floor, East Guanghua Tower\\
Fudan University, Shanghai}
\email{zhijiang@fudan.edu.cn}

\begin{document}
 \maketitle

\begin{abstract}
We study the canonical stability index of nonsingular projective varieties of general type with either large canonical volume or
large geometric genus. As applications of a general extension theorem established in the first part, we prove some optimal results in dimensions 
4 and 5, which are parallel to some well-known results on surfaces and 3-folds. 
\end{abstract}

\section{Introduction}

Understanding pluricanonical systems of nonsingular projective varieties of general type has been one of the major tasks in birational geometry.  By the work of Bombieri \cite{Bom} for surfaces and that of Hacon-McKernan \cite{H-M}, Takayama \cite{Tak}, Tsuji \cite{Tsu1} for higher dimensional varieties, for each $n>0$, there exists an optimal constant $r_n\in \mathbb{Z}_{>0}$, such that the pluricanonical map $\varphi_{m, X}$ is birational onto its image for all $m\geq r_n$ and for all nonsingular projective $n$-folds  $X$ of general type. Such a constant $r_n$ is, in general, non-explicit except that $r_1=3$, that $r_2=5$ by Bombieri and that $27\leq r_3\leq 61$ by Iano-Fletcher \cite{Iano-F} and Chen-Chen \cite[Theorem 1.6 (1)]{EXP3}. Usually the number $r_n$ is referred to as {\it the $n$-th canonical stability index}.

The existence of $r_n$ directly implies the existence of another optimal constant $v_n>0$ such that the canonical volume $\vol(X)\geq v_n$ for all nonsingular projective $n$-folds $X$ of general type, where we recall that the canonical volume ``$\vol(X)$'' is defined as the following number
$$\vol(X)=\underset{m\in \mathbb{Z}_{>0}}{\text{lim\ sup}}\frac{n!\ h^0(X,mK_X)}{m^n}$$
which is an important birational invariant.  We know that $v_1=2$ and $v_2=1$.
By Chen-Chen \cite[Theorem 1.6 (2)]{EXP3}, one has $v_3\geq \frac{1}{1680}$.  Similar to the situation of $r_n$, when $n\geq 4$, $v_n$ is non-explicit either.

It is also interesting to consider another optimal constant $r_n^+$ so that,  for all nonsingular projective $n$-folds $X$ of general type with $p_g(X)>0$, $\varphi_{m,X}$ is birational for all $m\geq r_n^+$.  By definition $r_n^+\leq r_n$ for any $n>0$.  One has $r_1^+=r_1$ and $r_2^+=r_2=5$ according to Bombieri. By Iano-Fletcher \cite{Iano-F}, Chen \cite{Ch03} and Chen-Chen \cite[Corollary 1.7]{EXP3}, we know $14\leq r_3^+\leq 18<r_3$. Recently, Brown and Kasprzyk \cite{B-K} found many new canonical 4-folds from which one deduces that $r_4>r_4^+\geq 39$.

In this paper we restrict our interest to varieties of large birational invariants. Early in 1973, Bombieri \cite{Bom} proved that, for a nonsingular projective surface $S$ with either the canonical volume $\vol(S)\geq 3$ or the geometric genus $p_g(S)\geq 4$,
$\varphi_{m, S}$ is birational for $m\geq r_1=r_1^+=3$. Such phenomenon may not be accidental, since the following theorem respectively due to G. Todorov and the first author has been realized:

\begin{thm}  Let $X$ be a nonsingular projective 3-fold of general type satisfying either of the following conditions:
\begin{itemize}
\item[(1)] $\vol(X)\geq 4355^3$ (see Todorov \cite{Tod}) ;
\item[(2)]  $p_g(X)\geq 4$ (see Chen \cite[Theorem 1.2 (2)]{Ch03}).
\end{itemize}
Then $\varphi_{m,X}$
is birational for all $m\geq r_2=r_2^+=5$
\end{thm}
Note that Theorem 0(1)  has been improved by the first author \cite{Ch} in loosing the volume constraint: say,
$\vol(X)\geq 12^3$ is sufficient.  One may also refer to Di Biagio \cite{DB} and Xu \cite{Xu} for other relevant results in dimensions 3 and 4.

Hence it is natural to consider the higher dimensional analog of Theorem 0. In fact such problem was first raised by McKernan in his review to the paper \cite{Tod} (see MR2339333 (2008g:14061)).  Anyway one should be basically guided by the following example:

\begin{exam}\label{LiZi} By the definition of $r_n$, for any $n\geq 3$, there is a nonsingular projective $(n-1)$-fold, say $X_{n-1}$, of general type so that $\varphi_{r_{n-1}-1, X_{n-1}}$ is not birational
by the definition of $r_{n-1}$. Pick up any smooth projective curve $C$ with $g( C)\geq 2$. Take $X_n=X_{n-1}\times C$.
Then $\varphi_{r_{n-1}-1, X_n}$ is not birational and, clearly, $\vol(X_n)$ can be arbitrarily large as long as $g(C)$ is arbitrarily large. Thus $r_n\geq r_{n-1}$. For the same reason, one knows that $r_n^+\geq r_{n-1}^+$.
\end{exam}

One of the main observations of this paper is the following result.

\begin{theo}\label{extension}
 Let $\mathfrak{X}$ be a birationally bounded family (see Definition 3.1) and $m>1$ be an integer.
Let $f:X\lra T$ be a morphism with connected fibers from a nonsingular projective variety $X$ of dimension $n\geq 2$ onto a smooth complete curve $T$. Assume that the general fiber $F$ of
$f$ is birationally equivalent to an element of $\mathfrak{X}$. Then there exists a constant $C(\mathfrak{X})>0$ such that, whenever $\vol(X)>C(\mathfrak{X})$,  the
restriction map $$H^0(X, mK_X)\rightarrow H^0(F_1, mK_{F_1})\oplus H^0(F_2, mK_{F_2})
$$ is surjective for any two different general fibers $F_1$ and $F_2$ of $f$.
\end{theo}

Theorem \ref{extension} has been applied to prove the following theorem.

\begin{theo}\label{main} For any integer $n>3$, there exists a constant $K(n)>0$ such that, for all nonsingular projective $n$-folds $X$ with $\vol(X)>K(n)$, the pluricanonical map $\varphi_m$ is birational for all
$$m\geq \text{max}\{r_{n-1}, \rounddown{2\prod_{j=2}^{n-1}\Big(1+(n-j)\sqrt[n-j]{\frac{2}{v_{n-j}}}\ \Big)}\}.$$
\end{theo}

As the 4-dimensional analog of Theorem 0(1), the following theorem is obtained by a direct application of Theorem \ref{main}.

\begin{theo}\label{1.2} There is a constant $K(4)>0$. For any smooth projective 4-folds $X$ with $\vol(X)>K(4)$, $\varphi_{m,X}$ is birational for all $m\geq r_3$.
\end{theo}

For projective varieties of large geometric genus, we prove the following theorem which is parallel to Theorem 0(2).

\begin{theo}\label{g}  There exist two constants $L(4)>0$ and $L(5)>0$. For any smooth projective $n$-fold  $X$ of general type with $p_g(X)>L(n)$ ($n=4, 5$), $\varphi_{m,X}$ is birational for all $m\geq r^+_{n-1}$.
\end{theo}


\begin{rema} The effectivity of $K(4)$, $L(4)$ and $L(5)$ relies on the constant $C(\mathfrak{X})$ for certain birationally bounded families $\mathfrak{X}$. 
\end{rema} 

We briefly explain the structure of this paper. In the first part,
we mainly prove the key extension theorem (i.e. Theorem \ref{extension}) which is the core of the paper. 
Then we use Theorem \ref{extension} to prove a general result like Theorem \ref{main}. In particular, we obtain
Theorem \ref{1.2} in dimension 4 which is parallel to Todorov's theorem for 3-folds. In Section 5, we study nonsingular projective varieties
with large geometric genus. Theorem \ref{g} generalizes an earlier theorem of the first author for 3-folds to the case of dimension $4$ and $5$. In the last section,  
we propose some open problems and conjectures and discuss their relations.

\section{Preliminaries and Notation}
\subsection{Notation} Throughout we adopt the following symbols:
\begin{itemize}
\item[$\circ$] ``$\sim$'' denotes linear equivalence.
\item[$\circ$] ``$\sim_{\mathbb Q}$'' denotes ${\mathbb Q}$-linear equivalence.
\item[$\circ$] ``$\equiv$'' denotes numerical equivalence.
\item[$\circ$] For $\bQ$-divisors $A$ and $B$,  ``$A\geq  B$'' means that $A-B$ is $\bQ$-linearly equivalent to an effective $\bQ$-divisor.
\item[$\circ$] ``$|M_1|\succeq |M_2|$'' means, for linear systems $|M_1|$ and $|M_2|$,
$|M_1|\supseteq |M_2|+(\text{effective divisor}).$
\item[$\circ$] $\text{Fix}|D|$ denotes the fixed part of the complete linear system $|D|$ and $\text{Mov}|D|=|D|-\text{Fix}|D|$ is the moving part. \end{itemize}

\subsection{\bf Convention.}   (1) For any linear system $|D|$ of positive dimension on a normal projective variety,  denote by $\Phi_{|D|}=\Phi_{\text{Mov}|D|}$ the rational map corresponding  to $|D|$.   In particular, $\varphi_{m,X}=\Phi_{|mK_X|}$ where $K_X$ denotes the canonical divisor of $X$.  

(2) We say that {\it $|D|$ is not composed of a pencil} if $\dim \overline{\Phi_{|D|}(X)}>1$.

(3) A {\it generic irreducible element} of $|D|$ means a general member of $\text{Mov}|D|$ when $|D|$ is not composed of a pencil or, otherwise, an irreducible component in a general member of $\text{Mov}|D|$. 

\subsection{Multiplier ideal sheaves and asymptotic multiplier ideals} We mainly refer to Lazarsfeld \cite[Part 3]{Laz2}. Here we provide the precise definition for a couple of ordinary forms of multiplier ideal sheaves which will appear in the context.
\begin{enumerate}
\item For an ideal sheaf $\fra\subseteq \OO_X$ and an effective $\mathbb{Q}$-divisor $D$ on a smooth projective variety $X$.  Let $\pi: X'\rightarrow X$ be a log resolution for the pair  $(\fra, D)$ such that the inverse ideal $\pi^{-1}(\fra)=\fra\cdot \cO_{X'}=\cO_{X'}(-E)$ for a divisor $E$ on $X'$ and that $E+\pi^*D$ has simple normal crossing supports. Then, for any real number $c>0$, the multiplier ideal $\cJ(X, \fra^c+D)$ is defined to be $$\pi_*\cO_{X'}(K_{X'/X}-\lfloor cE+\pi^*D\rfloor)$$
which is proven to be an ideal sheaf.

Especially, when $\fra=\OO_X$,  we obtain the multiplier ideal
$$\cJ(X,D)=\cJ(X, \OO_X+D). $$

\item For a line bundle $L$ on $X$ with non-negative Kodaira-Iitaka dimension, define ``$\frb_{\cdot}$'' to be the system of base ideals $\frb_k=\frb(|kL|)$ (where $k\in \mathbb{Z}_{>0}$)--the ideal sheaves corresponding to $\text{Bs}|kL|$. Then, for any positive real number $c>0$, the asymptotic multiplier ideal is defined to be
$$\cJ(X, ||cL||+D)=\cJ(X, \frb_p^{\frac{c}{p}}+D)$$ for any sufficiently large and sufficiently divisible integer $p$. \end{enumerate}

\subsection{The volume of a line bundle and the local multiplicity}
Let $Z$ be a normal projective $n$-fold and $D$ a Cartier divisor on $Z$.  The volume of $D$ is defined as:
$$\vol(D)=\underset{m\mapsto \infty}{\text{lim sup}} \frac{n!h^0(Z, \OO(mD))}{m^n}.$$
The canonical volume $\vol(Z)$ is defined to be $\vol(K_{Z'})$, where $Z'$ is any smooth projective model of $Z$.
We will frequently and tacitly use the following well-known results, where the integer $k$ can be independent of any given point.

\begin{lemm}\label{x} (see for instance Lazarsfeld \cite[Lemma 10.4.12]{Laz2})  Let $x\in Z$ be a smooth point. If $\vol(D)>\alpha^n$ for
some rational number $\alpha>0$, then, for any sufficiently divisible integer $k\gg 0$, there exists an effective divisor
$A_x\in |kD|$ with $\mult_x(A_x)>k\alpha$.
\end{lemm}

\begin{lemm}\label{xy} (see for instance Todorov \cite[Lemma 2.3]{Tod})  Let $x,y$ be two smooth points on $Z$. If $\vol(D)>2\beta^n$ for some rational number $\beta>0$, then, for any sufficiently divisible integer $k\gg 0$, there exists an effective divisor $A_{x,y}\in |kD|$ with $\mult_x(A_{x,y})>k\beta$ and $\mult_y(A_{x,y})>k\beta$.
\end{lemm}

\subsection{Non-klt locus} Let $D$ be an effective $\bQ$-divisor on a nonsingular projective variety $X$. The {\it non-klt locus}
$$\text{Nklt}(X,D)=\text{Supp}(\OO_X/\cJ(X,D))\subset X$$
can be endowed with the reduced scheme structure as a sub-scheme of $X$.

Let $\Delta$ be an effective $\bQ$-divisor on $X$ and $(X,\Delta)$ a log canonical pair.  For any point $x\in X$, we denote by $\text{LLC}(X,\Delta,x)$ the set of all log canonical centers passing through $x$.  The following result of Kawamata (\cite[Section 1]{Kaw1}) will be tacitly used in the proof of our main theorem. 

\begin{lemm}\label{unit}  Let $X$ be a nonsingular projective variety and $\Delta\geq 0$ an effective $\bQ$-divisor on $X$.  Assume that $(X,\Delta)$ is log canonical at some point $x\in X$.  The following statement hold:
\begin{itemize}
\item[(i)] If $W_1$, $W_2\in \text{LLC}(X,\Delta,x)$ and $W$ is any irreducible component of $W_1\cap W_2$ containing $x$, then $W\in \text{LLC}(X,\Delta,x)$. Therefore, if $(X,\Delta)$ is not klt at $x$, then $\text{LLC}(X,\Delta,x)$ has a unique minimal irreducible element, say $V$.
\item[(ii)] There exists an effective $\bQ$-divisor $E$ such that
$$\LLC(X, (1-\epsilon)\Delta+\epsilon E, x)=\{V\}$$
for all $0<\epsilon\ll 1$.
\item[(iii)] One may also assume that there is a unique place lying over $V$.  If $x$ is general and $D$ is a big divisor, then one can take  an effective $\bQ$-divisor $E\sim_{\mathbb Q} aD$, for some positive rational number $a$, to fulfill (ii).
\end{itemize}
\end{lemm}

\subsection{A weak extension theorem} The problem of 
extending pluricanonical forms is an important subject in birational geometry (see for instance \cite{siu, Kaw}). In this paper, we will always use the following weak version,  which corresponds to the special case (i.e. $S=\text{spec}(k)$) of Kawamata \cite[Theorem A]{Kaw}:

\begin{theo}\label{kaw-extension}
Let $V$ be a smooth projective variety and $X$ a smooth divisor of $V$. 
Assume that, for an ample $\bQ$-divisor $A$ and an effective $\bQ$-divisor $B$,  $K_V+X\sim_{\bQ} A+B$  and that $X$ is not contained in the support of $B$. Then
\begin{itemize}
\item[(1)] the natural restriction map $$H^0(V, m(K_V+X))\rightarrow H^0(X, mK_X)$$ is surjective for any integer $m\geq 2$;
\item[(2)] the natural restriction map $$H^0(V, m(K_V+L+X))\rightarrow H^0(X, mK_X+mL|_X)$$ is surjective for any nef divisor $L$ on $X$ and for any integer $m\geq 2$.
\end{itemize}
\end{theo}

\begin{rema} By \cite{BCHM} or \cite{Siu}, the pair $(V,X)$ has a minimal model $(V_0,X_0)$ since, for $0<\varepsilon_1\ll \varepsilon_2\ll 1$,  one may choose an effective   $\bQ$-divisor 
$$D_{\varepsilon_1,\varepsilon_2}\sim_{\bQ} (1-\varepsilon_1)X+(\varepsilon_1X+\varepsilon_2A)+\varepsilon_2B$$
so that  $(V,D_{\varepsilon_1,\varepsilon_2})$ is klt and one has
$$K_V+D_{\varepsilon_1,\varepsilon_2}\equiv (1+\varepsilon_2)(K_V+X).$$
Hence Theorem \ref{kaw-extension}(1) follows from a direct application of Kawamata-Viehweg vanishing theorem on $(V_0,X_0)$. Similarly one gets Theorem \ref{kaw-extension}(2) 
as well.
\end{rema}

\section{Multplier ideals and the restriction of pluricanonical forms}

We start with recalling the following definition (see \cite{HMX}).
\begin{defi}
We say that a set $\mathfrak{X}$ of varieties is birationally bounded if
there is a projective morphism between schemes, say $\tau: Z\rightarrow T$ where $T$ is of finite type, such that for
every element $X\in\mathfrak{X}$, there is a closed point $t\in T$ and a birational equivalence $X\dashrightarrow Z_t=\tau^{-1}(t)$.
\end{defi}

For a real number $M>0$, define $\mathfrak{X}_{n, M}$ to be the set of smooth projective $n$-folds $X$ of general type with $\vol(K_X)\leq M$. One knows that $\mathfrak{X}_{n, M}$ is birationally bounded (see \cite{H-M, Tak, Tsu1}).

\begin{prop}\label{prop1}
Let $\fra$ be an ideal sheaf on a smooth projective variety $X$ and $L$ an  effective $\mathbb{Q}$-divisor on $X$ with $L\not\equiv 0$. Fix a positive integer  $N>0$,  there exists a constant $c$, which depends on $L$, $N$ and $X$,  such that 
 $$\cJ(X, \fra^s+\epsilon D)=\cJ(X, \fra^s)$$
holds for any positive rational number $s$ whose denominator divides $N$,  for any number $\epsilon$ with $0\leq \epsilon<c$ and for any effective $\bQ$-divisor $D$ satisfying $D\equiv L$.  
\end{prop} 
\begin{proof}We first take a log resolution $\mu: X'\rightarrow X$ of $(X, \fra)$ and write $\mu^*(\fra)=\cO_{X'}(-\sum_ia_iE_i)$ where $\sum E_i$ is a $\mu$-exceptional divisor.

 By the birational transformation rule \cite[Theorem 9.2.33]{Laz2},  $$\cJ(X, \fra^s)=\mu_*\Big(\cO_{X'}(K_{X'/X})\otimes \cJ(X', \mu^*\fra^s)\Big)$$ and $$\cJ(X, \fra^s+\epsilon D)=\mu_*\Big(\cO_{X'}(K_{X'/X})\otimes \cJ(X', \mu^*\fra^s+\epsilon \mu^*D)\Big).$$  Hence, it suffices to show that $\cJ(X', \mu^*\fra^s)=\cJ(X', \mu^*\fra^s+\epsilon \mu^*D)$.

 We have
 \begin{eqnarray*}
 \cJ(X', \mu^*\fra^s)&=&\cO_{X'}(-\sum_i\rounddown{s a_i}E_i), \\
  \cJ(X', \mu^*\fra^s+\epsilon \mu^*D)&=&\cO_{X'}(-\sum_i\rounddown{s a_i}E_i)\otimes \cJ(X', \sum_i\{sa_i\}E_i+\epsilon\mu^*D).
  \end{eqnarray*}

  Note that, since $s=\frac{t}{N}$ with $t\in \mathbb{Z}$,  $\{sa_i\}\leq \frac{N-1}{N}$.
   Therefore, it suffices to prove the following statement:
 \begin{quote}
\noindent{ Let $\sum_ib_iE_i$ be a divisor with simple normal crossing support on a smooth projective variety $X$, where $0<b_i \leq \frac{N-1}{N}$.
Let $L$ be an effective $\mathbb{Q}$-divisor. Fix a very ample divisor $H$ on $X$ and let $M=(H^{n-1}\cdot L)_X$.
Then for all effective $\bQ$-divisor $D$ with $D\equiv L$ and for all $\epsilon<c=\frac{1}{NM}$, we have $\cJ(X, \sum_ib_iE_i+\epsilon D)=\cO_X$.}
\end{quote}

 Since $D\equiv L$ and $H$ is very ample,  $\mult_xD\leq M$ for any $x\in X$.  We shall apply a similar argument to that in the proof of \cite[Proposition 9.5.13]{Laz2}.  In fact, the above statement is certainly true in dimension $1$.  Assume $\dim X=n>1$. For any $x\in X$, we take a smooth divisor $Y\in |H|$  passing through $x$, which is not contained in the support of $D$, so that  $\sum_iE_i+H$ is a simple normal crossing divisor. Note that $H\mid_Y$ is still very ample and $(H\mid_Y^{n-2}\cdot D\mid_Y)_Y=M$. Then,  by induction, we have $$\cJ(Y, \sum_ib_iE_{i}\mid_Y+\epsilon D\mid_Y)=\cO_Y,$$ for all $\epsilon <c$. By the restriction theorem (see \cite[Theorem 9.5.1]{Laz2}), we have $$\cJ(Y, \sum_ib_iE_{i}\mid_Y+\epsilon D\mid_Y)\subset \cJ(X, \sum_ib_iE_i+\epsilon D)\mid_Y.$$
 Thus $\cJ(X, \sum_ib_iE_i+\epsilon D)_x=\cO_{X,x}$ holds for any $x\in X$. Hence $$\cJ(X, \sum_ib_iE_i+\epsilon D)=\cO_{X},$$ for all $\epsilon<c$.
 \end{proof}

 In practice we need a relative version of the above proposition.

 \begin{prop}\label{prop2}
 Let $f: \mathbb{X}\rightarrow T$ be a projective smooth morphism between irreducible smooth varieties.
 Let $N>0$ be a positive integer and let $\fra$ be an ideal sheaf on $\mathbb{X}$ such that $\fra_t=\fra\cdot \cO_{X_t}$ is non-zero
 for all $t\in T$. Let $L$ be an effective $\mathbb{Q}$-divisor on $\mathbb{X}$ with $L\not\equiv 0$  such that $X_t=f^{-1}(t)$ is not contained in the support of $L$ for each $t\in T$.  Then there exists a constant $c$, which depends on $L$,  $N$ and $\mathbb{X}$, such that
 $$\cJ(X_t, \fra_t^s+\epsilon D)=\cJ(X_t, \fra_t^s)$$ 
holds for all positive rational number $s$ whose denominator divides $N$,  for all $0\leq \epsilon<c$, for all $t\in T$, and for all effective $\bQ$-divisor $D$ on $X_t$ with $D\equiv L\mid_{X_t}$. 
 \end{prop}
\begin{proof}
Let $\mu: \mathbb{X}'\rightarrow \mathbb{X}$ be a log resolution of $(\mathbb{X}, \fra)$ and write $\mu^*(\fra)=\cO_{\mathbb{X}'}(-\sum_ia_iE_i)$ where $\sum E_i$ is a $\mu$-exceptional divisor. We note that there exists a Zariski open subset $U$ of $T$ such that $\sum_iE_i$ remains  a simple normal crossing divisor on $X'_t$ for all $t\in U$ (see \cite[page 211, proof of Theorem 9.5.35]{Laz2}).  Since $h$ is projective, we can fix a $h$-very ample divisor $H$ on $X$. Define $\hat{M}=(H\mid_{X_t}^{n-1}\cdot L\mid_{X_t})$ and define $C_1=\frac{1}{NM}$. As in the proof of Proposition \ref{prop1}, $\cJ(X_t, \fra_t^s+\epsilon D)=\cJ(X_t, \fra_t^s)$ for all $t\in U$ and for all $0\leq \epsilon <C_1$.

We then consider the finite number of families over $T\setminus U$. By Noether's induction, we may eventually find the desired positive number $c$ by taking the minimum of thos numbers in $\{C_i\}$.
\end{proof}

 \begin{prop}\label{prop3} Let $X$ be a nonsingular projective variety with positive Kodaira dimension. Assume that,
 for some integer $m>1$ and for a non-negative rational number $\epsilon<1$, the relation  $$\cJ(X, ||(m-1-\epsilon)K_X||+\epsilon D)\supseteq \cJ(X, ||(m-1)K_X||)$$ holds for all effective $\mathbb{Q}$-divisors $D$ with $D\equiv K_X$.
 Then, for all nonsingular projective birational model $X'$ of $X$ and for all effective $\mathbb{Q}$-divisor $D'$ on $X'$ with $D'\equiv K_{X'}$,
 the natural inclusion
\begin{eqnarray*}&&H^0(X', \cO_{X'}(mK_{X'})\otimes \cJ(X', ||(m-1-\epsilon)K_{X'}||+\epsilon D'))\\&\subseteq& H^0(X', \cO_{X'}(mK_{X'}))\end{eqnarray*} induces an isomorphism of vector spaces.
 \end{prop}
 \begin{proof}
 Since $\cJ(X, ||(m-1-\epsilon)K_X||+\epsilon D)\supseteq \cJ(X, ||(m-1)K_X||)$, we have
 \begin{eqnarray*}
&&H^0\big(X, \cO_X(mK_X) \big)\\
&\supseteq & H^0\Big(X, \cO_X(mK_X)\otimes\cJ\big(X, ||(m-1-\epsilon)K_X||+\epsilon D\big)\Big)\\ &\supseteq& H^0\Big(X, \cO_X(mK_X)\otimes\cJ\big(X, ||(m-1)K_X||\big)\Big)\\
 &\supseteq& H^0\Big(X, \cO_X(mK_X)\otimes\cJ\big(X, ||mK_X||\big)\Big)\\
 &=& H^0(X, \cO_X(mK_X) \big)
 \end{eqnarray*}
 by \cite[Proposition 11.2.10]{Laz2}.

 For any smooth birational model $X'$ of $X$, we take a common smooth birational model $X''$ of both $X$ and $X'$,  say  $\pi: X''\rightarrow X$ and $\mu: X''\rightarrow X'$. For some effective $\bQ$-divisor $D'\equiv K_{X'}$,  we take $D''=\mu^*D'+K_{X''/X'}\equiv K_{X''}$.  We may assume that $sD''$ is an effective divisor for some integer $s>0$.  
 Since $sD''=sK_{X''}+P''$ for some divisor $P''\in \Pic^0(X'')$, we write $P''=\pi^*(P)=\pi^*(sQ)$ for
 some divisor $Q\in \Pic^0(X)$. Since $h^0(X'',sK_{X''}+P'')=h^0(X, s(K_X+Q))$, we have $|sK_{X''}+P''|=\pi^*|s(K_X+Q)|+sK_{X''/X}$.
 Thus there exists an effective $\bQ$-divisor $D$ on $X$ such that $D''=\pi^*(D)+K_{X''/X}$.  Note that
 $D$ is nothing but $\pi_*(D'')$. We also note that for any positive integer $N$, $|NK_{X''}|=\pi^*|NK_X|+NK_{X''/X}$. Hence,
{\small  \begin{eqnarray*}
  &&\pi_*\Big(\cO_{X''}(mK_{X''})\otimes \cJ(X'', ||(m-1-\epsilon)K_{X''}||+\epsilon D'')
  \Big)\\
  &=& \pi_*\Big(\cO_{X''}(mK_{X''})\otimes \cJ(X'', ||(m-1-\epsilon)\pi^*K_{X}||+(m-1-\epsilon)K_{X''/X}+\epsilon D'')  \\
  &=&\pi_*\Big(\cO_{X''}(K_{X''}+(m-1)\pi^*K_{X})\otimes \cJ(X'', \pi^*||(m-1-\epsilon)K_{X}||+\epsilon \pi^*D)\Big)\\
&=&\cO_X(mK_X)\otimes\pi_*\Big(\cO_{X''}(K_{X''/X})\otimes \cJ(X'', \pi^*||(m-1-\epsilon)K_{X}||+\epsilon \pi^*D))\Big)\\
&=& \cO_X(mK_X)\otimes \cJ\big(X, ||(m-1-\epsilon)K_X||+\epsilon D\big),
 \end{eqnarray*}} where the first equality holds by
\cite[Proposition 9.2.31]{Laz2} and the definition of asymptotic multiplier ideals, the third equality holds by projection formula, and the last equality holds by the birational transformation rule \cite[Theorem 9.2.33]{Laz2}.

Hence
\begin{eqnarray*}
&&\dim H^0\Big(X'', \cO_{X''}(mK_{X''})\otimes \cJ(X'', ||(m-1-\epsilon)K_{X''}||+\epsilon D'')\Big)\\
&=&P_m(X)=P_m(X'')
\end{eqnarray*}
and we have the desired isomorphism on $X''$.

Similarly, we have
\begin{eqnarray*}
&&\mu_*\Big(\cO_{X''}(mK_{X''})\otimes \cJ(X'', ||(m-1-\epsilon)K_{X''}||+
\epsilon D'')\Big)\\
&=&\cO_{X'}(mK_{X'})\otimes \cJ\big(X', ||(m-1-\epsilon)K_{X'}||+\epsilon D'\big).
\end{eqnarray*} Thus
\begin{eqnarray*}
&&P_{m}(X')=P_m(X'')\\
&=&\dim H^0\Big(X', \cO_{X'}(mK_{X'})\otimes \cJ(X', ||(m-1-\epsilon)K_{X'}||+\epsilon D')\Big),
\end{eqnarray*} which means that the isomorphism on $X'$ holds.

The last equality naively implies the following property:
\begin{equation}\label{cdcd}
\mathfrak{b}(|mK_{X'}|)\subseteq \cJ(X', ||(m-1-\epsilon)K_{X'}||+\epsilon D').\end{equation} 
We are done.
\end{proof}

 Combing the above  propositions, we have the following theorem.

 \begin{theo}\label{theo1}
 Let $\mathfrak{X}$ be a birationally bounded set of smooth projective varieties. Then there exists a positive constant
 $c(\mathfrak{X})$ such that, for any nonsingular projective variety $X$ which is birationally equivalent to some element
 in $\mathfrak{X}$, for all effective $\mathbb{Q}$-divisors $D$ on $X$ with $D\equiv K_X$ and for all $0\leq \epsilon<c(\mathfrak{X})$, the isomorphism of vector spaces $$H^0\Big(X, \cO_X(mK_X)\otimes \cJ(X, ||(m-1-\epsilon)K_X||+\epsilon D)\Big)\cong H^0(X, \cO_X(mK_X))$$ holds for all $m\geq 2$.
 \end{theo}
 \begin{proof} Let $h: Z\rightarrow T$ be a projective morphism, where $T$ is of finite type,
 such that each element of $\mathfrak{X}$ is birational to a fiber of $h$.
 After taking birational modifications of $Z$ and $T$ and taking stratifications of $T$,
 we may assume that $h$ is a smooth morphism between smooth varieties, say $T=\sqcup_iT_i$ is of
 finite type, $T_i$ are affine,  and $h=\sqcup_i(h_i: Z_i\rightarrow T_i)$.  

 By \cite{BCHM}, we know that the canonical ring 
 $$\cR_i=\bigoplus_{m=0}^{\infty}h_{i*}\cO_{Z_i}(mK_{Z_i})$$ is a finitely generated $\cO_{T_i}$ graded algebra. Let $s_j\in h_{i*}\cO_{Z_i}(m_jK_{Z_i})$ ($1\leq j\leq k$) be generators of this ring and let $M\geq 2$ be a common multiple of all $m_j$.  Denote by $\tilde{\frb}_i$  the base ideal of $|MK_{Z_i}|$. By Siu's deformation invariance of plurigenera (see for instance \cite[Theorem 11.5.1]{Laz2}), $P_m(Z_t)$ is a constant for each $m\geq 1$ 
where $Z_t$ denotes the fiber of $h_i$ over $t\in T_i$. Equivalently, as being pointed out in \cite[p.310,(11.26), (11.27)]{Laz2}, the natural restriction (as a ring homomorphism between graded rings) $\cR_i\rightarrow \bigoplus_{m=0}^{\infty}H^0(Z_t, mK_{Z_t}) $ is surjective in all degrees for any $t\in T$. In particular, the natural map $h_{i*}\cO_{Z_i}(MK_{Z_i})\rightarrow H^0(Z_t, MK_{Z_t})$ is surjective.  We have the following commutative diagram:
$$\begin{CD}
H^0(MK_{Z_i})\otimes \OO_{Z_i}@  >\theta_3>> \OO_{Z_i}(MK_{Z_i})\otimes \tilde{\frb}_i\\
@V\theta_1VV  @ VV\theta_2V\\
H^0(MK_{Z_t})\otimes \OO_{Z_t} @>>\theta_4> \OO_{Z_t}(MK_{Z_t})\otimes \tilde{\frb}_i\otimes \OO_{Z_t}\end{CD}$$
where $\theta_1$ surjective and so is $\theta_2=\theta_1\otimes \tilde{\frb}_i$.  
Also $\theta_3$ is surjective by the definition of base ideal.  Thus $\tilde{\frb}_{i, t}=\tilde{\frb}_i\otimes_{\cO_{Z_i}}\cO_{Z_t}$ is the base ideal of $|MK_{Z_t}|$.
Hence $M$ is also a common multiple of the degrees of the generators of the canonical ring of $Z_t$. 

Now, by the definition of asymptotic multiplier ideal,  we have
\begin{eqnarray} 
\cJ(Z_t, ||(m-1)K_{Z_t}||)&=& \cJ\big(Z_t,  \frb(|(m-1)pK_{Z_t}|)^{\frac{1}{p}}\big) \ \text{for}\ p\gg 0 \nonumber\\
&=& \cJ\big(Z_t, \frb(|MK_{Z_t}|)^{\frac{l}{p}}\big)\ \ (here, (m-1)p=lM)\nonumber\\
&=& \cJ\big(Z_t, {\tilde{\frb}_{i,t}}^{\frac{m-1}{M}}\big). \nonumber\end{eqnarray}

Proposition \ref{prop2} implies that
 there exists a constant $\tilde{C}_i$ such that $\cJ(Z_t, ||(m-1)K_{Z_t} ||+\epsilon D_t)=\cJ(Z_t, ||(m-1)K_{Z_t}||)$ holds
 for all $t\in T_i$, for all effective $\bQ$-divisor $D_t$ on $Z_t$ with $D_t\equiv K_{Z_t}$ and for all $0\leq \epsilon<\tilde{C}_i$.  Hence
\begin{eqnarray}
\cJ(Z_t, ||(m-1-\epsilon)K_{Z_t}||+\epsilon D_t)
&\supseteq& \cJ(Z_t, ||(m-1)K_{Z_t} ||+\epsilon D_t)\nonumber \\
&=&\cJ(Z_t, ||(m-1)K_{Z_t}||).\nonumber\end{eqnarray} 

Since we only need to consider finitely many strata of $T$, we take $c(\mathfrak{X})=\min\{\tilde{C}_i\}$, which is positive. Then
$$\cJ(Z_t, ||(m-1-\epsilon)K_{Z_t}+\epsilon D_t||\supseteq \cJ(Z_t, ||(m-1)K_{Z_t}||)$$ holds
 for all $m\geq 2$, for all $t\in T$, for all effective $\bQ$-divisor $D_t$ on $Z_t$ with $D_t\equiv K_{Z_t}$ and for all $0\leq \epsilon<c(\mathfrak{X})$.
The main statement directly follows from Proposition \ref{prop3}.
 \end{proof}

 \begin{coro}\label{trivial}
 Let $\mathfrak{X}$ be a birationally bounded set of smooth projective varieties. Assume $P_m(Y)\neq 0$ for all $Y\in \mathfrak{X}$ and $m>1$. Then,
 for any nonsingular projective variety $X$ which is birationally equivalent to some element
 in $\mathfrak{X}$, for all effective $\mathbb{Q}$-divisors $D$ on $X$ with $D\equiv K_X$ and for all $0\leq \epsilon<c(\mathfrak{X})$, we have   $$\cJ(X, \epsilon D)\supseteq \mathfrak{b}(|mK_X|).$$
 \end{coro}
 \begin{proof}
 By Theorem \ref{theo1} and Property (\ref{cdcd}) in the proof of Proposition \ref{prop3}, we have 
$$ \cJ(X, \epsilon D)\supseteq\cJ(X, ||(m-1-\epsilon)K_X||+\epsilon D)\supseteq \mathfrak{b}(|mK_X|).$$
\end{proof}

\begin{theo}\label{cc} Let $\mathfrak{X}$ be a birationally bounded family.
Let $f:X\lra T$ be fibration from a nonsingular projective variety onto a smooth complete curve $T$. Assume that the general fiber $F$ of
$f$ is birationally equivalent to an element of $\mathfrak{X}$. Then there exists a constant $c_1(\mathfrak{X})>0$ such that, whenever $\vol(X)>c_1(\mathfrak{X})$,  the
restriction map $$H^0(X, mK_X)\rightarrow H^0(F, mK_F)$$ is surjective for all $m\geq 2$.
\end{theo}
\begin{proof} By assumption, $X$ must be of general type since $\vol(X)>0$. For $k\gg 0$, we have $P_k(X)\approx \vol(X)\cdot \frac{k^n}{n!}$ where $n=\dim X$.  On the other hand, there is a constant  $M(\mathfrak{X})>0$ such that $\vol(Z)< M(\mathfrak{X})$ for any element $Z\in\mathfrak{X}$. In particular, for the general fiber $F$ of $f$, $h^0(F, kK_F)< M(\mathfrak{X})\cdot \frac{k^{n-1}}{(n-1)!}$ for $k\gg 0$.
Pick general fibers $F_i$ ($1\leq i\leq s$) of $f$, we consider the short exact sequence
$$0\rightarrow \cO_X(kK_X-\sum_{i=1}^sF_i)\rightarrow \cO_X(kK_X)\rightarrow \oplus_i\cO_{F_i}(kF_i)\rightarrow 0.$$
By comparing dimensions of those cohomological groups in question, we see that $H^0(X, \cO_X(kK_X-\sum_{i=1}^sF_i))\neq 0$ for $k\gg 0$ and
$$\rounddown{\frac{k\vol(K_X)}{nM(\mathfrak{X})}}-1<s<\frac{k\vol(K_X)}{nM(\mathfrak{X})}.$$ Thus, $K_X\sim_{\bQ} pF+E$, where $E$ is an effective $\mathbb{Q}$-divisor and
$$p>\rounddown{\frac{\vol(K_X)}{nM(\mathfrak{X})}}-1$$ is a rational number.

 When $\vol(K_X)>2nM(\mathfrak{X})$, we can pick $p>1$, then $(m-1)K_X-F$ is a big divisor. We denote by $\cJ=\cJ(||(m-1)K_{X}-F||)$ and consider the short exact sequence:
\begin{equation} 0\rightarrow \cO_X(mK_{X}-F)\otimes\cJ\rightarrow \cO_X(mK_{X})\otimes\cJ\rightarrow \cO_{F}(mK_{F})\otimes\cJ |_{F}\rightarrow0.\nonumber\end{equation}
By Nadel's vanishing theorem, we have $$H^1(X, \cO_X(mK_{X}-F)\otimes\cJ)=0.$$
We then have the surjective map $$H^0(X, \cO_X(mK_{X})\otimes\cJ)\rightarrow H^0(F, \cO_{F}(mK_{F})\otimes\cJ |_{F}).$$
It suffices to show that the natural inclusion
$$H^0(F, \cO_{F}(mK_{F})\otimes\cJ |_{F})\subseteq H^0(F, \cO_{F}(mK_{F}))$$ is an isomorphism of vector spaces. Next we study the sheaf $\cJ|_{F}$.  Since $|l(m-1)K_F|\supseteq |l\big((m-1)K_X-F\big)|\mid_F$ for any integer $l>0$, we have 
$$\cJ|_F  \subseteq \cJ(||(m-1)K_{F}||)$$ (see for instance \cite[Section 11.2]{Laz}).

We pick a sufficiently large and divisible integer $N$ and
take $\epsilon=\frac{m}{p+1}$. Then
\begin{eqnarray*}
|N\Big((m-1)K_X-F\Big)|&=&|N\Big((m-1-\epsilon)K_X+(m-1-\epsilon)F+\epsilon E\Big)|\\
&\supseteq & |N(m-1-\epsilon)(K_X+F)|+N\epsilon E.
\end{eqnarray*}
On the other hand, by Theorem \ref{kaw-extension}, we know that
$$|N(m-1-\epsilon)(K_X+F)|\mid_F=|N(m-1-\epsilon)K_F|.$$
Hence $$\cJ\mid_F\supseteq \cJ(F, ||(m-1-\epsilon)K_F||+\epsilon D_F),$$
where $D_F=E\mid_F\sim_{\mathbb{Q}}K_F$ is an effective $\mathbb{Q}$-divisor.

Therefore \begin{eqnarray*}
&&H^0(F, \cO_{F}(mK_{F})\otimes\cJ |_{F})\\
&\supseteq& H^0(F, \cO_F(mK_F)\otimes \cJ(F, ||(m-1-\epsilon)K_F||+\epsilon D_F)).\end{eqnarray*}

As in the proof of Theorem \ref{theo1},  there exists a finite set $\mathbb{S}$ of positive integers such that  for any $F\in \mathfrak{X}$, the canonical ring of $F$ is generated by some elements whose degree numbers belong to $\mathbb{S}$. Let $N$ be a common multiple of those numbers in $\mathbb{S}$.
 When $\vol(X)$ is large enough, we can take a sufficiently large integer $p$ so that, for any $2\leq m\leq N+1$, $\epsilon=\frac{m}{p+1}\leq \frac{N+1}{p+1}<c( \mathfrak{X})$. Then, by Theorem \ref{theo1}, we see that the restriction map $H^0(X, mK_X)\rightarrow H^0(F, mK_F)$ is surjective for  $2\leq m\leq N+1$. On the other hand, for any $k\geq N+2$ and $s\in H^0(F, kK_F)$, $s$ can be written as a sum of products of elements of degrees between $2$ and $N$. Hence the restriction map $H^0(X, kK_X)\rightarrow H^0(F, kK_F)$ is again surjective and we conclude the statement.
 \end{proof}

\begin{coro}\label{restriction2}
Under the same condition of Theorem \ref{cc},  there exists another constant $c_2(\mathfrak{X})$ such that,
whenever $\vol(X)>c_2(\mathfrak{X})$, the restrction map $$H^0(X, mK_X)\rightarrow H^0(F_1, mK_{F_1})\oplus H^0(F_2, mK_{F_2})$$ is surjective for two arbitrary different general fibers $F_1$ and $F_2$ of $f$ and for all $m\geq 2$.
\end{coro}
\begin{proof}  Note that we have
$$K_X\sim_{\bQ} \hat{p}(F_1+F_2)+\hat{E}$$
with $\hat{E}\geq 0$ and $\hat{p}$ can be arbitrarily large as long as $\vol(X)$ is sufficiently large.  As in the proof of Theorem \ref{cc}, we study $\hat{\cJ}=\cJ(||(m-1)K_X-F_1-F_2||)$ instead of considering $\cJ$. All other argument follows accordingly without any problems. We omit the details and leave it to interested readers.
\end{proof}

\begin{rema}\label{diff} Corollary \ref{restriction2}  says that, as soon as $P_m(X)>0$ for some $m\geq 2$ and $\vol(X)>c_2(\mathfrak{X})$, $\varphi_m$ distinguishes different fibers of $f$.
\end{rema}

\subsection*[Proof of Theorem \ref{extension}]{Proof of Theorem \ref{extension}} Theorem \ref{cc} and Corollary \ref{restriction2} directly imply Theorem \ref{extension}.\qed

\section{Proof of Theorem \ref{main} and Theorem \ref{1.2}}

Given a projective variety $X$, by convention, we say that a point $x\in X$ {\it is in very general position} if $x$ is in a countably dominant subset, which is nothing but the complementary set of a union of countably many closed subsets of $X$.

\subsection{Point separation principle} Let $X$ be a nonsingular projective variety and $L$ a big divisor on $X$. Assume that $x$ and $y$ are two distinct very general points of $X$. If $D$ is an effective $\bQ$-divisor such that $L-D$ is nef and big, that $x,y\in \Nklt(X,D)$ and  that either $x$ or $y$ is an isolated lc center of $\Nklt(X,D)$, then Nadel's vanishing theorem implies that $|K_X+L|$ separates $x$ and $y$.

In practice, we are going to apply this principle to the special case with $L=(m-1)K_X$ for some integer $m>1$.  Write $K_X\sim_{\bQ} A+E$ where $A$ is an ample $\bQ$-divisor, $\vol(X)\approx\vol(A)$ and $E$ is an effective $\bQ$-divisor. Suppose we can find an effective $\bQ$-divisor $D\sim_{\bQ} aA$ for some rational number $a>0$ such that $x$, $y\in \Nklt(X,D)$ and that either $x$ or $y$ is an isolated lc center. Then it follows that $\varphi_{m, X}$ is birational for all $m\geq \rounddown{a}+2$.

\subsection{The induction formula according to Takayama}\label{kaya}
We review some formulae and results, of Takayama \cite[Section 5]{Tak}, which will be used in our proof.  {}First of all, we fix the notation following \cite[Notation 5.2]{Tak}.  Let $X$ be a nonsingular projective $n$-fold.
\begin{itemize}
\item[$\circ$] For any integer $d$ with $1\leq d\leq n$, assume that $\alpha_d>0$ is a constant so that
$\vol(Z)>\alpha_d^{n-d}$ holds for any sub-variety $Z$ of codimension $d$ passing through very general points.

\item[$\circ$] Pick any rational number $\varepsilon$ with $0<\varepsilon\ll 1$.

\item[$\circ$] Non-negative integers $s_d$, $s_d'$, $t_d$ and $t_d'$ are defined inductively as follows. Let $s_1=0$ and $t_1=n \sqrt[n]{2}/(1-\varepsilon)$. Assume that $s_d$ and $t_d$ are defined already. Set $s_d'=s_d+\varepsilon$, $t_d'=t_d$ and
\begin{eqnarray*}
s_{d+1}&=& \Big(1+\frac{\sqrt[n-d]{2} (n-d)}{(1-\varepsilon)\alpha_d}\Big)s_d'+
\frac{2\sqrt[{n-d}]{2}(n-d)}{(1-\varepsilon)\alpha_d},\\
t_{d+1}&=&\Big(1+\frac{\sqrt[n-d]{2}(n-d)}{(1-\varepsilon)\alpha_d}\Big)t_d'.
\end{eqnarray*}


\item[$\circ$] For the given constant $\varepsilon$, we may find a birational modification $\mu:X'\lra X$ with $X'$ smooth and take the decomposition
$$\mu^*(K_X)\sim_{\bQ} A_{\varepsilon}+E_{\varepsilon}$$
where $A_{\varepsilon}$ is $\bQ$-ample and $E_{\varepsilon}$ is an effective $\bQ$-divisor.  For simplicity we just write $A=A_{\varepsilon}$ and $E=E_{\varepsilon}$ as no confusion is likely. Note also that we may assume $|\vol(A)-\vol(X)|<\varepsilon^{100}$ as $\varepsilon$ is very small.
\end{itemize}

The following proposition of Takayama provides an effective induction:

\begin{prop}\label{T}(cf.  \cite[Proposition 5.3]{Tak}) Let $x_1$, $x_2$ be two different very general points on $X'$. The following statement $(*_d)$ holds for every $d$ with $1\leq d\leq n$:
\begin{quote}{$(*_d)$} There exist a positive constant $a_d<s_d+t_d/\alpha_0$, a non-empty subset $I_d\subset \{1,2\}$ and an effective $\bQ$-divisor $D_d$ on $X'$ satisfying $D_d\sim_{\bQ} a_dA$ such that
\begin{itemize}
\item[(i)] $(X',D_d)$ is log canonical at $x_i$ for all $i\in I_d$;
\item[(ii)] $(X',D_d)$ is not klt at $x_i$ for each $i\in I_d$;
\item[(iii)]  $(X',D_d)$ is not log canonical at $x_j$ for each $j\in \{1,2\}\setminus I_d$;
\item[(iv)] $\codim \Nklt(X', D_d)\geq d$ at $x_i$ for all $i\in I_d$.
\end{itemize}
\end{quote}
\end{prop}

In fact, Takayama's clue of the proof of Proposition \ref{T} is:
$$(*_d)\Longrightarrow (*_d')\Longrightarrow (*_{d+1}),$$
where $``(*_d')"$ is the following statement:
\begin{quote}  $(*_d')$ There exist a positive constant $a_d'<s_d'+t_d'/\alpha_0$, a non-empty subset $I_d'\subset \{1,2\}$ and an effective $\bQ$-divisor $D_d'$ on $X'$  with $D_d'\sim_{\bQ} a_d'A$ such that
\begin{itemize}
\item[(i)] $(X', D_d')$ is log canonical at $x_i$ for all $i\in I_d'$,
\item[(ii)] $(X',D_d')$ is not klt at $x_i$ for each $i\in I_d'$,
\item[(iii)] $(X', D_d')$ is not log canonical at $x_j$ for each $j\in \{1,2\}\setminus I_d'$,
\end{itemize}
and either of the following is true:
\begin{itemize}
\item[(iv-0)] $\codim \Nklt(X', D_d')>d$ at $x_i$ for all $i\in I_d'$.
\item[(iv-1)] $\Nklt(X',D_d')=Z\cup Z_{+}$ such that $Z$ is irreducible of codimension $d$, and $x_i\in Z$ but $x_i\not\in Z_{+}$ for $i\in I_d'$. In particular $Z$ is unique.
\end{itemize}
\end{quote}

Since Proposition \ref{T} is not strong enough for our purpose, we need to use the following improved version for which a similar inequality was observed by Di Biago \cite[Theorem 5.5]{DB}. We observed as well that Debarre mentioned this in his survey article (see \cite[Theorem 5.2]{Deb}).

\begin{vari}\label{variant} Keep the same notation as in Proposition \ref{T}. Assume that there is a constant $\tilde{a}_d>0$ and an effective $\bQ$-divisor $\tilde{D}_d\sim_{\bQ}\tilde{a}_dA$ so that the pair $(X', \tilde{D}_d)$ satisfies conditions $(*_d): (i)\sim (iv)$. Then
\begin{itemize}
\item[(I)] there exists an effective $\bQ$-divisor $\tilde{D}_d'\sim_{\bQ}\tilde{a}_d'A$ for some positive constant
$$\tilde{a}_d'<\tilde{a}_d+\varepsilon$$ such that $(X', \tilde{D}_d')$ satisfies conditions
$(*_d')$.

\item[(II)] there exists an effective $\bQ$-divisor $\tilde{D}_{d+1}\sim_{\bQ} \tilde{a}_{d+1}A$ for some positive constant
\begin{eqnarray}\label{IDF}{\ta}_{d+1}<{\ta}_d\Big(1+\frac{\sqrt[n-d]{2} (n-d)}{\alpha_d}\Big)+\frac{2\sqrt[n-d]{2}(n-d)}{\alpha_d}+\varepsilon,
\end{eqnarray}
such that $(X', \tilde{D}_{d+1})$ satisfies conditions $(*_{d+1}): (i)\sim (iv)$.
\end{itemize}
\end{vari}
\begin{proof}  Statement (I) follows directly from the proof of \cite[Lemma 5.5]{Tak}.
One keeps the same construction for $\tilde{D}_d'$  as that for $D_d'$ by replacing $D_d$ with $\tilde{D}_d$,  $a_d$ with $\tilde{a}_d$ and so on. In fact,
one gets $\tilde{a}_d'<\tilde{a}_d+\varepsilon$ just consulting \cite[line 25, page 580]{Tak} and \cite[line 15, page 581]{Tak}.

For Statement (II), we first refer to the proof of \cite[Main Lemma 5.6]{Tak}, where we observe that $b=\rounddown{\ta_d'}+1$ and thus the main consequence of the proof of \cite[Main Lemma 5.6]{Tak} should be:
$$(Z\cdot A^{n-d})>\Big(\frac{(1-\varepsilon)\alpha_d}{1+b}\Big)^{n-d}.$$
Then we refer to the proof of \cite[Lemma 5.8]{Tak}, where we may replace $b_d(X)$, $D_d'$, $D_{d+1}$ and $a_{d+1}$ by $b$, $\tilde{D}_d'$, $\tilde{D}_{d+1}$ and $\ta_{d+1}$, respectively.  The same argument, especially by the definitions of $a_{d+1}$ in \cite[line 3, line 19, page 583]{Tak}, implies
\begin{eqnarray*}
\ta_{d+1}&<&\tilde{a}_d'+\frac{\sqrt[n-d]{2} (n-d)(1+b)}{(1-\varepsilon)\alpha_d}\\
&=&{\ta}_d\Big(1+\frac{\sqrt[n-d]{2} (n-d)}{\alpha_d}\Big)+\frac{2\sqrt[n-d]{2}(n-d)}{\alpha_d}+f(\varepsilon)
\end{eqnarray*}
where $f(\varepsilon)\mapsto 0$ (as $\varepsilon\mapsto 0$).  We are done.
\end{proof}

By Proposition \ref{T}, we can find an effective $\bQ$-divisor $D_1$ with $D_1\sim_{\bQ} a_1A$,  such that $(X',D_1)$ satisfies the condition $(*_1)$ where
$$a_1<\frac{\sqrt[n]{2}n}{\alpha_0}+\varepsilon.$$
For any integer $m$ satisfying $1\leq m\leq n$, Variant \ref{variant}(II) inductively implies that one can find a rational number
\begin{equation}\label{am}
a_m<\Big(\frac{\sqrt[n]{2}n}{\alpha_0}+2\Big)\prod_{j=1}^{m-1}\Big(1+\frac{\sqrt[n-j]{2}(n-j)}{\alpha_j}\Big)-2+\varepsilon
\end{equation}
for  the given $\varepsilon\ll 1$
and an effective $\bQ$-divisor $D_m$ with $D_m\sim_{\bQ} a_mA$ such that $(X', D_m)$ satisfies the condition $(*_m)$.
\medskip

A direct corollary of Takayama's induction and Variant \ref{variant} is as follows.

\begin{coro}\label{trivialbound}
Let $V_n=\Big(\sqrt[n]{2}n+2\Big)\prod_{j=1}^{n-1}\Big(1+\frac{\sqrt[n-j]{2}(n-j)}{\nu_j}\Big)+1$. Then for all $n$-fold $X$
with $\vol(X)\geq 1$, the $m$-th pluricanonical map $\varphi_m$ is birational for all $m\geq V_n$. In particular, for all nonsingular projective 3-folds of general type with $\vol(X)\geq 1$, $\varphi_{38,X}$ is stably birational. 
\end{coro}

Now we are ready for proving the main theorem. We shall discuss the case of dimension 4 in details while the proof for higher dimensions is simply a redundant generalization.

\subsection{Proof of Theorem \ref{main} (the case of dimension 4)}\label{dim4}  Let $X$ be a nonsingular projective 4-fold of general type.  Given a number $\varepsilon$ with $0<\varepsilon\ll 1$. Modulo possibly a birational modification, we may and do assume that $K_X\sim_{\bQ} A+E$ for some ample $\bQ$-divisor $A$ with $0<\vol(X)- \vol(A)\ll \varepsilon$ and an effective $\bQ$-divisor $E$ on $X$.  Assume that $\vol(X)>\alpha_0^4$ for some rational number $\alpha_0>0$.  Actually we will be working under certain assumption that $\alpha_0$ is sufficiently large.
Note that we always have $\alpha_3\geq 2$ and $\alpha_2\geq 1$.
Pick two very general points $x_1$, $x_2\in X$.

By Lemma \ref{xy}, Lemma \ref{unit} and Proposition \ref{T}, we may take an effective $\bQ$-divisor $D_0=D_0^{(x_1,x_2)}\sim_{\mathbb{Q}} a_0A$ with $a_0<\frac{4\sqrt[4]{2}}{\alpha_0}+\varepsilon$ so that
$(X,D_0)$ satisfies the condition $(*_1)$ of Proposition \ref{T}: namely,  both $x_1$ and $x_2$ belong to $\Nklt(X,D_0)$;   $(X,D_0)$ is {\it lc} at either $x_1$ or $x_2$, we may simply  assume that $(X,D^{(x_1,x_2)}_0)$ is {\it lc} at $x_1$ modulo an exchanging of indices, and there exists a unique {\it lc} center $V_{x_1}$ passing through $x_1$.

We then consider the Hilbert scheme $\cH_a$, for a rational number $a<\frac{4\sqrt[4]{2}}{\alpha_0}+\varepsilon$, parametrizing the data
$$\{(x_1, x_2, D_0^{(x_1,x_2)})\mid x_1, x_2\in D_0^{(x_1,x_2)}\sim_{\mathbb{Q}}  aA\}.$$
 
 Note that $\cH_a$ is the subset of $X\times X\times |MA|$, where $M$ is an integer such that $MA$ is a divisor. The set  $\cH_a$ consists of $\{(x_1, x_2, D)\}\in X\times X\times |MA|$ such that $D$ passes through both $x_1$ and $x_2$.
Pull back the universal divisors in $|MA|\times X$ for each $M$, we get the universal divisor $\cD\subset \cH_a\times X$,   parametrized by $\cH_a$.  Take a log resolution of the pair $(\cH_a\times X, \cD)$. Since log canonical singularities is defined by the discrepancies of a log resolution of a pair, we see easily that, in each component of $\cH_a$ ($a$ is a small rational number), there exists a constructible subset parametrizing those datum such that 
$$(X, D_0^{(x_1,x_2)})=(X, \frac{a}{M}D)$$ is {\it non-klt} at $x_1$, $x_2$ and is {\it log canonical} at $x_1$ with a unique {\it log canonical} centre $V_{x_1}$ passing through $x_1$. Then, by the countability of Hilbert schemes $\{\cH_a\}$, we see that there exists a rational number $a_0<\frac{4\sqrt[4]{2}}{\alpha_0}+\varepsilon$ and an irreducible scheme $U$ parametrizing the data
$(x_1, x_2, V_{x_1}, D_0^{(x_1,x_2)})$ with the property: $D_0^{(x_1,x_2)}\sim_{\mathbb{Q}} a_0A$; $(X, D_0^{(x_1, x_2)})$ is non-klt at $x_1$ and $x_2$; $(X, D_0^{(x_1, x_2)})$  is {\it lc} at $x_1$ with unique {\it lc} centre $V_{x_1}$ passing through $x_1$; the natural morphism
\begin{eqnarray*}U &&\rightarrow X\times X\ \text{defined by}\\
(x_1, x_2, V_{x_1}, D_0^{(x_1,x_2)}) &&\rightarrow (x_1, x_2)\end{eqnarray*}
is dominant. Shrinking $U$ a little bit, we may obtain the minimal number $r_0$ with $0\leq r_0\leq 3$ such that,  for very general pair $(x_1, x_2)\in X\times X$, there exists $(x_1, x_2, V_{x_1}, D_0^{(x_1,x_2)})$, parametrized by $U$, admitting $\dim V_{x_1}=r_0$.
\medskip

{\bf Case 1}.  Assume $r_0\leq 2$.

This means that the pair $(X,D_0^{(x_1,x_2)})$ satisfies the condition $(*_2)$ of Proposition \ref{T}. Set $\ta_2=a_0$. Repeatedly applying Variant \ref{variant}(II), we may find an effective $\bQ$-divisor $\tilde{D}_4\sim \ta_4A$ such that $(X, \tilde{D}_4)$ satisfies the condition $(*_4)$, where $\ta_4>0$ is a rational number with:
\begin{eqnarray}
\ta_4&<&2(1+2\sqrt{2})\ta_2+(8\sqrt{2}+2)+3\varepsilon\\
&<&(8\sqrt{2}+2)+\frac{8(1+2\sqrt{2})\sqrt[4]{2}}{\alpha_0}+(4\sqrt{2}+5)\varepsilon.
\label{k1'}\end{eqnarray}
Clearly one can find a computable constant $K_1>0$ such that the right hand side of Inequality (\ref{k1'})
is upper bounded by $\roundup{8\sqrt{2}+2}=14$ whenever $\alpha_0>K_1$.  Therefore, when $\vol(X)>K_1^4$,  $\varphi_m$ is birational for all $m\geq 15$.
\medskip

{\bf Case 2}.  Assume $r_0=3$.

Set $a_1=a_0$. According to Variant \ref{variant}(II) and Inequality (\ref{am}), we may find an effective $\bQ$-divisor $D_4\sim a_4A$ such that $(X, D_4)$ satisfies the condition ($*_4$), where
$$a_4<2(1+2\sqrt{2})\Big(1+\frac{3\sqrt[3]{2}}{\alpha_1}\big)\big(2+\frac{4\sqrt[4]{2}}{\alpha_0}\Big)-2+h(\varepsilon)$$
and $h(\varepsilon)\mapsto 0$.
\medskip

{\bf Step 1}.  If $\alpha_1>1000$, by taking a sufficiently small $\varepsilon$, then we have
$$4(1+2\sqrt{2})\Big(1+\frac{3\sqrt[3]{2}}{\alpha_1}\big)-2<14.$$
Thus one can find a computable constant $K_2>0$ such that $a_4<14$ whenever $\alpha_0>K_2$.  So, when $\vol(X)>K_2^4$ and $\alpha_1>1000$, $\varphi_m$ is birational for all $m\geq 15$.
\medskip

{\bf Step 2}. Slightly shrinking $U$ if necessary, we may assume that, for each $V_y$ parametrized by $U$, the number $\vol(V_y)\leq 1000^3$.  By McKernan's idea \cite{Mc} and particularly Todorov \cite[Lemma 3.2]{Tod} (see also \cite[P.1328]{Tod}), we have the surjective morphism $\pi:X''\lrw X$ from a nonsingular projective variety $X''$ onto $X$ and the following diagram:
$$\begin{CD}
X''@>\pi>> X\\
@VfVV &\\
B&
\end{CD}$$
where $f:X''\lrw B$ is a fibration onto the smooth curve $B$.
We organize the argument by distinguishing two cases:
\medskip

{\bf Subcase 2.1. $\pi$ is not birational}. Clearly, passing through a very general point $x_1$, there are at least two $lc$ centers, say $V_{x_1}, V_{y}$ corresponding to two elements of $U$.
According to Lemma \ref{unit}, one may find a rational number $\ta_2<3a_0+\varepsilon$ and an effective divisor $\tilde{D}_2\sim \ta_2A$ such that the pair $(X, \tilde{D}_2)$ satisfies condition $(*_2)$. Parallel to the situation of Case 1 and apply Variant \ref{variant}(II) once more, we may find an effective $\bQ$-divisor $\tilde{D}_4\sim \ta_4A$ such that $(X, \tilde{D}_4)$ satisfies the condition $(*_4)$, where $\ta_4>0$ is a rational number with:
\begin{eqnarray}
\ta_4&<&2(1+2\sqrt{2})\ta_2+(8\sqrt{2}+2)+3\varepsilon\\
&<&(8\sqrt{2}+2)+\frac{24(1+2\sqrt{2})\sqrt[4]{2}}{\alpha_0}+(12\sqrt{2}+9)\varepsilon.
\label{k2'}\end{eqnarray}
One can find a computable constant $K_3>0$ such that the right hand side of Inequality (\ref{k2'}) is upper bounded by $\roundup{8\sqrt{2}+2}=14$ whenever $\alpha_0>K_3$.  Therefore, when $\vol(X)>K_3^4$,  $\varphi_m$ is birational for all $m\geq 15$.
\medskip

{\bf Subcase 2.2. $\pi$ is birational}.  For simplicity, we may simply assume $X''=X$. By construction, the fiber of $f$ passing through a very general point $x_1$ is exactly the center $V_{x_1}\subset \Nklt(X, D_{0}^{(x_1,x_2)}) $ with $D_{0}\sim a_0 A$ and $a_0<\frac{4\sqrt[4]{2}}{\alpha_0}+\varepsilon$.
By our assumption $\vol(V_{x_1})\leq 1000^3$.  Now we take $\mathfrak{X}=\mathfrak{X}_{3, 1000^3}$, which is a birationally bounded family. By Theorem \ref{cc} and Corollary \ref{restriction2},
there exists a constant $K_4$ such that, whenever $\vol(X)>K_4^4$,  $\varphi_{m,X}$ is birational for all $m\geq r_3$.

By Chen-Chen \cite{EXP3}, we know that $27\leq r_3\leq 61$. Take $K(4)=\text{max}\{K_1^4, K_2^4,K_3^4,K_4^4\}$. Then we have seen that $\varphi_{m,X}$ is birational whenever $\vol(X)>K(4)$. So we have proved the theorem in dimension 4.  \qed

The above argument has the following direct corollary.

\begin{coro}\label{444}  There exists a constant $\hat{K}(4) >0$. For any smooth projective 4-fold $X$ with $p_g(X)>0$ and $\vol(X)>\hat{K}(4)$, $\varphi_{m}$ is birational for $m\geq \text{max}\{15, r_3^+\}$.  In particular, $\varphi_{18,X}$ is birational.
\end{coro}
\begin{proof} Note that all argument in Subsection \ref{dim4} follows except Subcase 2.2 where $r_3$ should be replaced by $r_3^+$?as the general fiber of $f$ has positive geometric genus.

Since $14\leq r_3^+\leq 18$ as mentioned in Introduction, the statement follows.
\end{proof}

\subsection{Proof of Theorem \ref{main} (the general case)}
Assume $\dim(X)>4$. The argument is similar in the spirit to that of \ref{dim4}. Thus we will omit most of the parallel details.

Given a number $\varepsilon$ with $0<\varepsilon\ll 1$. We still write $K_X\sim_{\bQ} A+E$ for some ample $\bQ$-divisor $A$ with $0<\vol(X)- \vol(A)\ll \varepsilon$ and an effective $\bQ$-divisor $E$ on $X$. Pick two very general points $x_1$, $x_2\in X$.

By Lemma \ref{xy}, Lemma \ref{unit} and Proposition \ref{T}, we may take an effective $\bQ$-divisor $D_0=D_0^{(x_1,x_2)}\sim_{\mathbb{Q}} a_0A$ with $a_0<\frac{n\sqrt[n]{2}}{\alpha_0}+\varepsilon$ so that
$(X,D_0)$ satisfies the condition $(*_1)$ of Proposition \ref{T}. We may simply assume that $(X,D^{(x_1,x_2)}_0)$ is {\it lc} at $x_1$ modulo an exchanging of indices  and that there exists a unique {\it lc} center $V_{x_1}$ passing through $x_1$.

Similarly we consider the Hilbert scheme $\cH_a$, for $a<\frac{n\sqrt[n]{2}}{\alpha_0}+\varepsilon$, and obtain an irreducible scheme $U$ as in \ref{dim4}. We may still assume that the scheme $U$ allows to obtain a constant $r_0=\dim V_{x_1}$ where $(x_1, x_2, V_{x_1}, D_0^{(x_1,x_2)})\in U$. Clearly $0\leq r_0\leq n-1$.

 \medskip

{\bf Case I}.  Assume $r_0\leq n-2$.

This means that the pair $(X,D_0^{(x_1,x_2)})$ satisfies the condition $(*_2)$ of Proposition \ref{T}. Set $\ta_2=a_0$. Repeatedly applying Variant \ref{variant}(II), we may find an effective $\bQ$-divisor $\tilde{D}_n\sim \ta_nA$ such that $(X, \tilde{D}_n)$ satisfies the condition $(*_n)$, where $\ta_n>0$ is a rational number with:
\begin{equation}
\ta_n<\big(\frac{n\sqrt[n]{2}}{\alpha_0}+2\big)\prod_{j=2}^{n-1}\Big(1+
\frac{\sqrt[n-j]{2}(n-j)}{\alpha_j}\Big)-2+l(\varepsilon)
\label{kn1}\end{equation}
with $l(\varepsilon)\mapsto 0$ as $\varepsilon\mapsto 0$.
Clearly one can find a computable constant $\hat{K}_1>0$ such that the right hand side of Inequality (\ref{kn1}) is strictly smaller than
$u_1=\rounddown{2\prod_{j=2}^{n-1}\Big(1+
\frac{\sqrt[n-j]{2}(n-j)}{\alpha_j}\Big)}-1$ whenever $\alpha_0>\hat{K}_1$.  Therefore, when $\vol(X)>\hat{K}_1^4$,
$\varphi_{m,X}$ is birational for all $m\geq u_1+1$.
\medskip

{\bf Case II}.  Assume $r_0=n-1$.

Set $a_1=a_0$. According to Variant \ref{variant}(II) and Inequality (\ref{am}), we may find an effective $\bQ$-divisor
$D_n\sim a_nA$ such that $(X, D_n)$ satisfies the condition ($*_n$), where
\begin{equation}\label{n-1}
a_n<\Big(\frac{\sqrt[n]{2}n}{\alpha_0}+2\Big)\prod_{j=1}^{n-1}\Big(1+\frac{\sqrt[n-j]{2}(n-j)}{\alpha_j}\Big)-2+\varepsilon
\end{equation}
\medskip

{\bf Step i}. {\bf The case with $\alpha_1>T$}. {}From Inequality (\ref{n-1}), it is clear that there
are two constant $T>0$ and $\hat{K}_2>0$ such that, whenever $\alpha_1>T$ and $\alpha_0>\hat{K}_2$, the right side of Inequality (\ref{n-1})
is strictly smaller than $u_1$. Thus, meanwhile, $\varphi_{m,X}$ is birational for all $m\geq u_1+1$.
\medskip

{\bf Step ii}. {\bf The case with $\alpha_1\leq T$}. Slightly shrinking $U$ if necessary, we may assume that, for each $V_y$ parametrized by $U$, the number $\vol(V_y)\leq T^{n-1}$.  By Todorov \cite[Lemma 3.2]{Tod} (see also \cite[P.1328]{Tod}), we have the surjective morphism $\pi:X''\lrw X$ from a nonsingular projective variety $X''$ onto $X$ and the following diagram:
$$\begin{CD}
X''@>\pi>> X\\
@VfVV &\\
B&
\end{CD}$$
where $f:X''\lrw B$ is a fibration onto the smooth curve $B$.
We organize the argument by distinguishing two cases:
\medskip

{\bf Subcase ii.1. $\pi$ is not birational}. Clearly, passing through a very general point $x_1$, there are at least two $lc$ centers, say $V_{x_1}, V_{y}$ corresponding to two elements of $U$.
According to Lemma \ref{unit}, one may find a rational number $\ta_2<3a_0+\varepsilon$ and an effective $\bQ$-divisor $\tilde{D}_2\sim \ta_2A$ such that the pair $(X, \tilde{D}_2)$ satisfies condition $(*_2)$. Parallel to the situation of Case I and apply Variant \ref{variant}(II) once more, we may find an effective $\bQ$-divisor $\tilde{D}_n\sim \ta_nA$ such that $(X, \tilde{D}_n)$ satisfies the condition $(*_n)$, where $\ta_n>0$ is a rational number with (compare Inequality (\ref{kn1})):
\begin{equation}
\ta_n<\big(\frac{3n\sqrt[n]{2}}{\alpha_0}+2\big)\prod_{j=2}^{n-1}\Big(1+
\frac{\sqrt[n-j]{2}(n-j)}{\alpha_j}\Big)-2+h(\varepsilon)
\label{kn2}\end{equation}
with $h(\varepsilon)\mapsto 0$.
One can find a computable constant $\hat{K}_3>0$ such that the right hand side of Inequality (\ref{kn2}) is strictly smaller than $u_1$ whenever $\alpha_0>\hat{K}_3$.  Therefore, when $\vol(X)>\hat{K}_3^4$,  $\varphi_m$ is birational for all $m\geq u_1+1$.
\medskip

{\bf Subcase ii.2. $\pi$ is birational}.  For simplicity, we may simply assume $X''=X$. By construction, the fiber of $f$ passing through a very general point $x_1$ is exactly the center $V_{x_1}\subset \Nklt(X, D_{0}^{(x_1,x_2)}) $ with $D_{0}\sim a_0 A$ and $a_0<\frac{n\sqrt[n]{2}}{\alpha_0}+\varepsilon$.
By our assumption, $\vol(V_{x_1})\leq T^{n-1}$.  Now we take $\mathfrak{X}=\mathfrak{X}_{n, T^{n-1}}$, which is a birationally bounded family. By Theorem \ref{cc} and Corollary \ref{restriction2},
there exists a constant $\hat{K}_4$ such that, whenever $\vol(X)>\hat{K}_4^n$,  $\varphi_{m,X}$ is birational for all $m\geq r_{n-1}$.

{}Finally take $K(n)=\text{max}\{\hat{K}_1^n, \hat{K}_2^n,\hat{K}_3^n,\hat{K}_4^n\}$. What we have proved is that $\varphi_{m,X}$ is birational
for $$m\geq \text{max}\{r_{n-1}, u_1+1 \}$$
whenever $\vol(X)>K(n)$.  Since
$$\rounddown{2\prod_{j=2}^{n-1}\Big(1+(n-j)\sqrt[n-j]{\frac{2}{v_{n-j}}}\ \Big)}
= u_1+1, $$
the main theorem is proved.
\qed

\section{\rm Proof of Theorem \ref{g}}

\subsection{Convention}

Let $Z$ be any nonsingular projective variety of general type over an algebraically closed field $k$ of characteristic $0$. By \cite{BCHM} or \cite{Siu}, we know that the canonical ring $R(Z,K_Z)$ is finitely generated over $k$.
 One can obtain the morphism  $\pi_Z: Z\longrightarrow Z_0$ possibly after a birational modification of $Z$.

\subsection{The canonical map $\varphi_1$}\label{sset}  Let $X$ be a nonsingular projective $n$-fold of general type. Assume $p_g(X)\geq 2$.  Modulo further birational morphism, we may assume that $\text{Mov}|K_X|$ is base point free. Taking the Stein factorization, we have the following diagram:
\begin{eqnarray*}
\xymatrix{
X\ar[rr]^f\ar[drr]^{\varphi_{1}}\ar[d]_{\pi_X} && \Gamma \ar[d]^{s}\\
X_0&& \mathbb{P}^{N}
}
\end{eqnarray*}
where $\varphi_{1}$ is the canonical  morphism and $N=p_g(X)-1$. Denote by $H\in \text{Mov}|K_X|$ the general member. Then  $h^0(X, K_X-H)=1$.  Set $d=\dim \Gamma$.

\subsection{Canonical restriction inequalities}
 Denote by $F$ a general fiber of $f$. Unless $d=n$, $F$ is a smooth projective variety of dimension $n-d>0$.
 We then take $d-1$ general hyperplane sections $H_1, \cdots, H_{d-1}$  on $s(\Gamma)$. Set the curve $W=H_1\cap\cdots\cap H_{d-1}$, $X_W=f^{-1}(W)$ and $X_{H_i}=f^{-1}(H_i)$ for each $i$.  We may assume that $X_W$ and $X_{H_i}$ are all nonsingular for each $i$.  Clearly we have $X_{H_i}\sim H$ for each $i=1,\cdots, d-1$.

Consider the fibration $f_W: X_W\rightarrow W$ with $F$ a general fiber. First we write $H|_{X_W}=\sum_j F_j\equiv aF$ for some smooth fibers $F_j$ of $f$, where $a=H_1^d\geq p_g(X)-d$ (see for instance \cite[Lemma 1.3]{Kob}). By Theorem \ref{kaw-extension}, we have
 \begin{equation}\label{ee1}
 |m(K_{X_W}+\frac{1}{a}H|_{X_W})|\mid_F= |m(K_{X_W}+F)|\mid_F=|mK_F|
 \end{equation} for sufficiently large and divisible integer $m$.

Note also that, $K_{X_W}\sim(K_X+(d-1)H)\mid_{X_W}\leq  (dK_X)\mid_{X_W}$.    Moreover, for a sufficiently large and divisible integer $m$, 
Theorem \ref{kaw-extension}(2) implies:
\begin{eqnarray}&&|m(K_X+(d-1+\frac{1}{a})H)|\mid_{X_W}\cr
&=&|m(K_X+\sum_{1\leq i\leq d-1}X_{H_i}+\frac{1}{a}H)|\mid_{X_W}\nonumber
\\&=&|m(K_{X}+X_{H_1})+m(\sum_{2\leq i\leq d-1}X_{H_i}+\frac{1}{a}H)|\mid_{X_{H_1}}\mid_{X_W}\cr
&=& |m\Big(K_{X_{H_1}}+(\sum_{2\leq i\leq d-1}X_{H_i}+\frac{1}{a}H)\mid_{X_{H_1}}\Big)|\mid_{X_W}\cr
&& \vdots\nonumber\\
&=& |m\Big(K_{X_{H_1\cap\cdots\cap H_{d-2}}} +(X_{d-1}+\frac{1}{a}H)|_{X_{H_1\cap\cdots\cap H_{d-2}}} \Big)||_{X_{W}} \nonumber \\
&=& |m(K_{X_W}+\frac{1}{a}H|_{X_W})| \label{ee2}
\end{eqnarray}
where $X_{H_1\cap\cdots\cap H_{d-2}}=f^{-1}(H_1\cap\cdots\cap H_{d-2})$.
 Since $K_X-H$ is effective,  we conclude from (\ref{ee1}) and (\ref{ee2}) that
 $$|\frac{m(ad+1)}{a}K_X|\mid_{F}~\succeq ~|mK_F|.$$
By the base point free theorem, we have
\begin{eqnarray*}
\text{Mov}|\frac{m(ad+1)}{a}K_X|&=&|\pi_X^*\big(\frac{m(ad+1)}{a}K_{X_0}\big)|,\\
\text{Mov}|mK_F|&=&|\pi_F^*(mK_{F_0})|.
\end{eqnarray*}
Thus
\begin{equation}\pi_X^*(K_{X_0})\mid_F~\geq ~\frac{a}{1+ad}\pi_F^*(K_{F_0}). \label{cri}
\end{equation}

{}Finally, noting that $\pi_X^*(K_{X_0})\geq H$, we have
\begin{eqnarray}\label{ee3}
\vol(X)&=&(\pi_X^*(K_{X_0}))_X^n \geq (H^d\cdot \pi_X^*(K_{X_0})^{n-d})_X\nonumber \\
&=&a \cdot (\pi_X^*(K_{X_0})\mid_F)_F^{n-d}
\geq a\cdot (\frac{a}{1+ad})^{n-d}\cdot \pi_F^*(K_{F_0})^{n-d}\nonumber  \\
&=& a\cdot (\frac{a}{1+ad})^{n-d}\cdot \vol(F)\cr
&\geq &\frac{v_{n-d}}{(d+1)^{n-d}}\cdot (p_g(X)-d).
\end{eqnarray}

\begin{coro}\label{noether} There exist positive constants $a_n$ and $b_n$ such that $$\vol(X)\geq a_np_g(X)-b_n$$ holds for any smooth projective $n$-fold $X$ of general type.
\end{coro}
\begin{proof}  Keep above settings. When $d=n$, one has
$\vol(X)\geq 2p_g(X)-2n$ by \cite[Proposition (2.1)]{Kob}. Thus we have
\begin{equation}
\vol(X)\geq \underset{{1\leq d\leq n-1}}{\text{min}}\{2p_g(X)-2n, \frac{\upsilon_{n-d}}{(d+1)^{n-d}}(p_g(X)-d) \}.
\end{equation}
\end{proof}

\subsection{Key reduction steps}\label{step} We keep the same setting as in Subsection \ref{sset}. We always have an induced fibration $f=f_W: X_{W}\rw W$ with the general fiber $F$ a smooth projective $(n-d)$-fold of general type.  Assume $d<n$.
We may write
$$\pi_X^*(K_{X_0})|_{X_W}=\sum_{i=1}^{a} F_i+E_W\equiv aF+E_W$$
where $E_W$ is an effective $\bQ$-divisor on $X_W$, each $F_i$ is a smooth fiber of $f$ and $a\geq p_g(X)-d$.  For any integer $m$ with
$m>d+\frac{1}{a}$,  we may consider the linear system
\begin{equation}|K_X+\roundup{(m-d-\frac{1}{a})\pi_X^*(K_{X_0})}+(d-1)H|\preceq |mK_X|. \label{k1}\end{equation}
The $\bQ$-divisor $(m-d-\frac{1}{a})\pi_X^*(K_{X_0})|_{F}$ is nef and big and after further modification of $X$, we may and do assume that  $\pi_X^*(K_{X_0})$ has simple normal crossing  support.  Kawamata-Viehweg vanishing theorem (\cite{Kawa, V}) implies
\begin{equation}
|mK_X||_F\succeq |K_F+\roundup{(m-d-\frac{1}{a})\pi_X^*(K_{X_0})|_{F}}|. \label{g1}
\end{equation}

Our key theorem will essentially work under the following assumption:  

{\it $(*)$ Assume that, on a general fiber $F$,  there are $l=n-d-1$ positive integers, say $n_1$, $\cdots$, $n_l$,  such that, for each $i=1,\cdots,l$, $\dim\overline{\Phi_{|L_i|}(F)}\geq i$ where $L_i=\text{Mov}|\rounddown{n_i\pi_X^*(K_{X_0})|_F}|$.}
\medskip

Modulo further necessary birational modifications, we may and do assume that, for each $i$, $|L_i|$ is base point free.  For an integer $m>d+\frac{1}{a}+\sum_i n_i$, 
we have 
\begin{eqnarray}
&&|K_F+\roundup{(m-d-\frac{1}{a})\pi_X^*(K_{X_0})|_{F}}|\cr
&\succeq& 
|K_F+\roundup{(m-d-\frac{1}{a}-\sum_{i=1}^l n_i)\pi_X^*(K_{X_0})|_{F}}+\sum_{i=1}^l L_i|.
\end{eqnarray}

 Take $V^{(1)}$ to be a generic irreducible element of $|L_1|$ on $F$. Since $|L_1|$ is base point free, $V^{(1)}$ is smooth. For $j=2,\cdots, l$, take $V^{(j)}$ to be a generic irreducible element of $|{L_j}|_{V^{(j-1)}}|$. 
Each $V^{(j)}$ is smooth simply due to the base point freeness of $|L_j|$ for each $j$. Set $C=V^{(l)}$, which is a smooth complete curve by our assumption that $\dim\overline{\Phi_{|L_l|}(F)}\geq l$. 

By the Kawamata-Viehweg vanishing theorem theorem, for each $j=1,\cdots,l$, we have

\begin{eqnarray}
&& |K_F+\roundup{(m-d-\frac{1}{a}-\sum_{i=1}^l n_i)\pi_X^*
(K_{X_0})|_{F}}
+\sum_{i=1}^l L_i||_{V^{(j)}}\cr
&\succeq &|K_{V^{(j-1)}}+\roundup{(m-d-\frac{1}{a}-\sum_{i=1}^l n_i)\pi_X^*(K_{X_0})|_{V^{(j-1)}}}
+\sum_{i=j}^l {L_i}|_{V^{(j-1)}}||_{V^{(j)}}\cr
&\succeq&|K_{V^{(j)}}+\roundup{(m-d-\frac{1}{a}-\sum_{i=1}^l n_i)\pi_X^*(K_{X_0})|_{V^{(j)}}}
+\sum_{i=j+1}^l {L_i}|_{V^{(j)}}|.
\end{eqnarray}
Particularly, when $j=l$, we simply get
\begin{equation} |mK_X||_C\succeq |K_C+D_m|\label{k2}\end{equation}
where $D_m=\roundup{(m-d-\frac{1}{a}-\sum_{i=1}^l n_i)\pi_X^*(K_{X_0})|_C}$ is a divisor on $C$ with
$$\deg(D_m)\geq (m-d-\frac{1}{a}-\sum_{i=1}^l n_i)(\pi_X^*(K_{X_0})\cdot C).$$
Set $\eta=(\pi_X^*(K_{X_0})\cdot C)$ and
$\alpha_m=(m-d-\frac{1}{a}-\sum_{i=1}^l n_i)\eta$. Clearly, since $\pi_X^*(K_{X_0})$ is nef and big and $C$ is moving in an algebraic family, we see $\eta>0$. 
 
\begin{theo}\label{kk} Let  $m>1$ be an integer. Let $X$ be a nonsingular projective $n$-fold ($n\geq 4$) with $p_g(X)\geq 2$. Keep the same notation as above.   Assume $d<n$ and that Assumption (*) holds. Then
\begin{itemize}
\item[(1)]  under the condition $\alpha_m>1$, we have 
$$m\eta\geq \deg(K_C)+\roundup{(m-d-\frac{1}{a}-\sum_{i=1}^l n_i)\eta}.$$
Consequently one has $$\eta\geq \frac{\deg(K_C)}{d+\frac{1}{a}+\sum_{i=1}^l n_i}.$$
\item[(2)] $\varphi_{m,X}$ is birational provided that $\alpha_m>2$.
 \end{itemize}
\end{theo}

This theorem is in fact parallel to the key theorem in dimension 3  (see, for instance, Chen-Chen \cite[Subsection 2.10]{JDG}).  So we just explain the central part while omitting those redundant details. We will tacitly use the following lemma in reducing the birationality to lower dimensional case:
\medskip

\noindent{\bf Lemma A}. {\em Let $V$ be a smooth projective variety with $p_g(V)>0$. Let $Q$ be an effective nef and big 
$\bQ$-divisor on $V$. Assume that $|L|$ is a base point free linear system on $V$. Then $|K_V+\roundup{Q}+L|$ distinguishes different generic irreducible elements of $|L|$. In other words, $\Phi_{|K_V+\roundup{Q}+L|}$ is birational if and only if so is $\Phi_{|K_V+\roundup{Q}+L|}|_Y$ for each generic irreducible element $Y$ of $|L|$. }
\begin{proof}
When $|L|$ is not composed of an irrational pencil, since 
$|K_V+\roundup{Q}+L|\succeq |L|$, $\Phi_{|K_V+\roundup{Q}+L|}$ automatically distinguishes different irreducible elements of $|L|$. 

When $|L|$ is composed of an irrational pencil,  we may assume that $Q$ has simple normal crossing supports after a further modification of $X$. Pick two different generic irreducible elements, say $Y_1$ and $Y_2$, of $|L|$. Then $L-Y_1-Y_2$ is nef.  
Then it follows from the Kawamata-Viehweg vanishing theorem that one has the surjective map:
$$H^0(V, K_V+\roundup{Q}+L)\rightarrow 
H^0(Y_1, K_{Y_1}+\roundup{Q}|_{Y_1})\oplus H^0(Y_2, K_{Y_1}+\roundup{Q}|_{Y_2}), $$
where $H^0(Y_i, K_{Y_i}+\roundup{Q}_{Y_i})\neq 0$ for $i=1,2$. Thus Lemma A holds. \end{proof}

\subsection*{Proof of Theorem \ref{kk}} 
Since $p_g(X)>0$, frequent applications of Lemma A tells us that $\varphi_{m,X}$ is birational if and only if $\varphi_{m,X}|_C$ is birational for any irreducible component $C$ in $|L_l|_{V^{(l-1)}}|$. Hence we naturally refer to Relation (\ref{k2}). 

When $\deg(D_m)>1$, $|K_C+D_m|$ is base point free. Thus Relations (\ref{k1}) and (\ref{k2}) imply that
\begin{equation}m(\pi_X^*(K_{X_0})\cdot C)\geq \deg(K_C)+\deg (D_m)\label{mn}, \end{equation}
which proves (1).  One may always take a very large number $m'$ so that $\deg(D_{m'})>1$. Then Inequality (\ref{mn}) reads $\eta\geq \frac{\deg(K_C)}{d+\frac{1}{a}+\sum_{i=1}^l n_i}$ by eliminating $m'$.

When $\deg(D_m)>2$, $|K_C+D_m|$ gives a birational map and so is $\varphi_{m,X}$. So the theorem holds true.
\qed
\medskip

Applying the above method to the extremal case with $d=n$, we obtain the following corollary which might be known to experts. 

\begin{coro}\label{d=n} Let $X$ be a nonsingular projective $n$-fold ($n\geq 3$) such that $\varphi_{1,X}$ is generically finite map. Then $\varphi_{m,X}$ is birational for $m\geq n+2$.
\end{coro}
\begin{proof} We are in the situation $d=n$. So $f_W:X_W\rw W$ is a finite map where $\tilde{C}=X_W$ is a smooth projective curve.  
By the vanishing theorem, we have
\begin{eqnarray*}
|mK_X||_{\tilde{C}}&\succeq & |K_X+\roundup{(m-d)\pi_X^*(K_{X_0})}+H_{d-1}+\cdots+H_1||_{\tilde{C}}\\
&\succeq& |K_{\tilde{C}}+\roundup{(m-d)\pi_X^*(K_{X_0})|_{\tilde{C}}}|.
\end{eqnarray*}
Repeatedly using Lemma A, we can see that $\varphi_{m,X}$ is birational if and only if so is 
$\varphi_{m,X}|_{\tilde{C}}$ for a general member $\tilde{C}\in |X_W|$.

We have $\deg \pi_X^*(K_{X_0})|_{\tilde{C}}\geq 2$, since $\varphi_{1,X}$ is generically finite. 
So $|K_{\tilde{C}}+\roundup{(m-d)\pi_X^*(K_{X_0})|_{\tilde{C}}}|$ gives a birational map when $m-d\geq 2$. Thus $\varphi_m$ is birational for all $m\geq n+2$.
\end{proof}

\begin{coro}\label{d=n-1} Let $X$ be a nonsingular projective $n$-fold ($n\geq 3$) with  $\dim\overline{\varphi_{1,X}(X)}=n-1$. 
 Then $\varphi_{m,X}$ is birational for $m\geq 2n+1$.
\end{coro}
\begin{proof} Under the assumption, $F$ is a smooth projective curve of genus $\geq 2$. The situation fits into the consideration of Theorem \ref{kk}.
We have $C=F$, $n_i=0$ for each $i$, and $d=n-1$.

Pick a general fiber $F$ of $f$, we need to study the linear system
$$|K_F+\roundup{(m-d-\frac{1}{a})\pi_X^*(K_{X_0})|_F}|$$
by virtue of Relation (\ref{g1}).  According to 
Inequality (\ref{cri}), we have 
$$\deg~\pi_X^*(K_{X_0})|_F\geq \frac{2a}{a(n-1)+1}$$
noting that $F=F_0$ is a smooth complete curve of genus $\geq 2$. 

Whenever $m\geq 2n+1$, we have 
$$\alpha_m\geq (m-d-\frac{1}{a})~\pi_X^*(K_{X_0})|_F\geq  (n+2-\frac{1}{a})\cdot \frac{2a}{a(n-1)+1}>2.$$ Thus, by Theorem \ref{kk}, 
$\varphi_{m,X}$ is birational when $m\geq 2n+1$. We are done.
\end{proof}

We need the following simple lemma to treat the case with $d=n-2$. 

\begin{lemm}\label{X} Let $F$ be smooth projective surface of general type. Then, for any $\bQ$-divisor $Q_{\lambda}$ satisfying $Q_{\lambda}\equiv \lambda\pi_F^*(K_{F_0})$ with $\lambda>3$, $|K_F+\roundup{Q_{\lambda}}|$ gives a birational map.
\end{lemm}
\begin{proof} When $(K_{F_0}^2, p_g(F_0))\neq (1,2)$,  this is due to Chen-Chen \cite[Lemma 2.3, Lemma 2.5]{EXP3}.

When $(K_{F_0}^2, p_g(F_0))=(1,2)$,  $|K_F|$ is composed of a rational pencil of genus 2 curves. Modulo a further blow up, we may assume that $\text{Mov}|K_F|$ is base point free. 
Pick two different generic irreducible elements $C$ and $C_1$ of $|K_F|$. Write 
$$\pi_F^*(K_{F_0})\sim C+E_0$$
for an effective divisor $E_0$ on $F$. Then since
$$Q_{\lambda}-C-C_1-2E_0\equiv (\lambda-2)\pi_F^*(K_{F_0})$$
is nef and big, Kawamata-Viehweg vanishing theorem implies the surjective map:
$$H^0(F, K_F+\roundup{Q_{\lambda}-2E_0})\lrw
H^0(C,K_C+D)\oplus H^0(C_1, K_{C_1}+D_1)$$
where $D=(\roundup{Q_{\lambda}-2E_0-C})|_C$ and 
$D_1=(\roundup{Q_{\lambda}-2E_0-C_1})|_{C_1}$. Note that $|C|$ is a free pencil. Clearly we have 
$$\deg(D)\geq (\lambda-2) (\pi_F^*(K_{F_0})\cdot C)>0$$
and, similarly, $\deg(D_1)>0$. So $H^0(K_C+D)\neq 0$ and 
$H^0(C_1, K_{C_1}+D_1)\neq 0$. This means that $|K_F+\roundup{Q_{\lambda}}|$ distinguishes different irreducible elements of $|C|$. 

Applying Kawamata-Viehweg vanishing theorem, we have the surjective map
$$H^0(K_F+\roundup{Q_{\lambda}-E_0})\lrw H^0(K_C+
\roundup{Q_{\lambda}-E_0-C}|_C)$$
where 
$$\deg(\roundup{Q_{\lambda}-E_0-C}|_C)\geq (\lambda-1)(\pi_F^*(K_{F_0})\cdot C)=\lambda-1>2.$$
Thus $\Phi_{|K_F+\roundup{Q_{\lambda}-E_0}|}|_C$ is birational. So is $\Phi_{|K_F+\roundup{Q_{\lambda}}|}$. 
\end{proof}

\begin{coro}\label{d=n-2} Let $X$ be a nonsingular projective $n$-fold ($n\geq 3$) with  $\dim\overline{\varphi_{1,X}(X)}=n-2$. 
Then $\varphi_{m,X}$ is birational for $m\geq 4n-3$.
\end{coro}
\begin{proof} By assumption, $F$ is a smooth projective surface of general type and $p_g(F)>0$. By Inequality (\ref{cri}), we have $$\pi_X^*(K_{X_0})|_F\geq \frac{a}{a(n-2)+1}\pi_F^*(K_{F_0})$$
where $a\geq p_g(X)-d$. Referring to Relation (\ref{g1}), it is sufficient to study the linear system
$$|K_F+\roundup{(m-n+2-\frac{1}{a})\pi_X^*(K_{X_0})|_{F}}|.$$
Whenever $m\geq 4n-3$, we have
$$(m-d-\frac{1}{a})\pi_X^*(K_{X_0})|_{F}\equiv \lambda \pi_F^*(K_{F_0})+H_{\lambda}$$
for some rational number $\lambda>3$ and an effective $\bQ$-divisor $H_{\lambda}$. By Lemma \ref{X},
we see that $|K_F+\roundup{(m-d-\frac{1}{a})\pi_X^*(K_{X_0})|_{F}-H_{\lambda}}|$ gives a birational map. Hence $\varphi_{m,X}$ is birational when
$m\geq 4n-3$.
\end{proof}

\begin{rema} With the same method, Corollary \ref{d=n-1} and Corollary \ref{d=n-2} can be easily improved when $p_g(X)$ is large enough: 
\begin{itemize}
\item[(i)] When $\dim\overline{\varphi_{1,X}(X)}=n-1$ and $p_g(X)\geq n+2$, $\varphi_{2n-1,X}$ is birational.

\item[(ii)] When $\dim\overline{\varphi_{1,X}(X)}=n-2$ and $p_g(X)\geq n+3$, $\varphi_{4n-7,X}$ is birational. 
\end{itemize}
\end{rema}

\subsection[Proof of Theorem \ref{g}]{Proof of Theorem \ref{g} (the case $n=4$)}\label{n4}

By Corollary \ref{d=n}, Corollary \ref{d=n-1} and Corollary \ref{d=n-2}, we know that $\varphi_{m,X}$ is birational whenever $m\geq 13$ and $d\geq 2$. 
Hence we only need to study the case with $d=1$.

 In this case, we have $W=\Gamma$ and $X_W=X$. Pick a general fiber $F$. We may write
$$\pi_X^*(K_{X_0})\equiv pF+E$$
where $E$ is an effective $\bQ$-divisor and $p\geq p_g(X)-1$.
For any integer $m>1$, the vanishing  theorem implies
\begin{eqnarray}
|(m+1)K_X||_F&\succeq &|K_X+\roundup{(m-1)\pi_X^*(K_{X_0})}+F||_F\cr
&\succeq &|K_F+\roundup{(m-1)\pi_X^*(K_{X_0})|_F}|.\label{rest}
\end{eqnarray}

 {\bf Case 1}. $\vol(F)>12^3$.  By Theorem 0, $\varphi_{5,F}$ is birational. {}First let us consider the case with $g(\Gamma)>0$.  By Chen \cite[Lemma 2.5]{Ch10}, we have
$\pi_X^*(K_{X_0})|_F\sim \pi_F^*(K_{F_0})$.  In particular,  $\pi_X^*(K_{X_0})|_F$ is a Cartier divisor.  So Relation (\ref{rest}) implies that $\varphi_{6,X}|_F$ is birational. Hence $\varphi_{m,X}$ is birational for all $m\geq 6$.  Next let us consider the case with $g(\Gamma)=0$.  Suppose $p\geq 5$. Then Theorem \ref{kaw-extension}(1) implies that
$$|6K_X||_F\succeq |5(K_X+F)|_F=|5K_F|,$$
which means $\varphi_{6,X}$ is birational.  In a word, $\varphi_{6,X}$ is birational when  $p_g(X)\geq 6$ and $\vol(F)>12^3.$

{\bf Case 2}. $\vol(F)\leq 12^3$.  Take $\mathfrak{X}$ to be set of all those 3-folds of general type with positive geometric genus and the canonical volume being upper bounded by $12^3$. By Theorem \ref{extension}, there exists a constant $c(\mathfrak{X})>0$ such that, whenever $\vol(X)>c(\mathfrak{X})$, $\varphi_{m,X}$ is birational for all $m\geq r_3^+$.  Now by Corollary \ref{noether}, there exists a constant $L(4)>0$ such that $p_g(X)>L(4)$ implies $\vol(X)>c(\mathfrak{X})$.

Since we know $r_3^+\geq 14$ as mentioned earlier,  we have actually proved Theorem \ref{g} in the case $n=4$.
\qed

\subsection[Proof of Theorem \ref{g}]{Proof of Theorem \ref{g} (the case $n=5$)}
By Corollary \ref{d=n}, Corollary \ref{d=n-1} and Corollary \ref{d=n-2}, we know that $\varphi_{m,X}$ is birational whenever $m\geq 17$ and $d\geq 3$.
Hence we only need to study the cases $d=1, 2$.

{\bf Case I}. $d=2$.  In this case, $F$ is a smooth projective 3-fold of general type with $p_g(F)>0$.  By Relation (\ref{g1}), we need to study the $\bQ$-divisor
$Q_m=(m-d-\frac{1}{a})\pi_X^*(K_{X_0})|_F$ where $a\geq p_g(X)-d$.
When $p_g(X)\geq 20$, we have $a\geq 18$. For $m\geq 35$, Inequality (\ref{cri}) implies
$$Q_m\equiv (m-d-\frac{1}{a})\cdot \frac{a}{2a+1}\pi_F^*(K_{F_0})+H_m$$
where $(m-d-\frac{1}{a})\cdot \frac{a}{2a+1}>16$ and $H_m$ is an effective $\bQ$-divisor. Thus, by Chen-Chen \cite[Theorem 8.1]{EXP3}, $|K_F+\roundup{(m-d-\frac{1}{a})\pi_X^*(K_{X_0})|_F}|$ gives a birational map. Thus $\varphi_{35,X}$ is birational whenever $p_g(X)\geq 20$.

{\bf Case II}. $d=1$.  Similar to the discussion in Subsection \ref{n4},  we still have the parallel relation to (\ref{rest}):
\begin{equation}
|(m+1)K_X||_F\succeq |mK_X+F||_F\succeq |K_F+\roundup{(m-1)\pi_X^*(K_{X_0})|_F}|.\label{>0}\end{equation}

Subcase II.1.  $\vol(F)>\hat{K}(4)$ where  $\hat{K}(4)$ is the same constant as in Corollary \ref{444}.  When $g(\Gamma)>0$,  we have
$\pi_X^*(K_{X_0})|_F\sim \pi_F^*(K_{F_0})$
by Chen \cite[Lemma 2.5]{Ch10}.   Then both Relation (\ref{>0}) and Corollary \ref{444} imply that $\varphi_{19,X}$ is birational.  When $g(\Gamma)=0$, we work under the assumption $p_g(X)\geq 19$.   Then the extension theorem \ref{kaw-extension} implies that
$$|19K_X||_F\succeq |18(K_X+F)|_F=|18K_F|,$$
which means $\varphi_{19,X}$ is birational.  To make the summary, $\varphi_{19,X}$ is birational when  $p_g(X)\geq 19$ and $\vol(F)>\hat{K}(4)$.

Subcase II.2.  $\vol(F)\leq \hat{K}(4)$.  Take $\mathfrak{X}$ to be set of all those 4-folds of general type with positive geometric genus and the canonical volume being upper bounded by $\hat{K}(4)$. By Theorem \ref{extension}, there exists a constant $c(\mathfrak{X})>0$ such that, whenever $\vol(X)>c(\mathfrak{X})$, $\varphi_{m,X}$ is birational for all $m\geq r_4^+$.  Now by Corollary \ref{noether}, there exists a constant $L(5)>0$ such that $p_g(X)>L(5)$ implies $\vol(X)>c(\mathfrak{X})$ .

Since we know $r_4^+\geq 39$ as mentioned earlier,  we have proved Theorem \ref{g} in the case $n=5$.
\qed

\section{Open problems and further discussion}

Naturally we would like to propose the following conjectures:

\begin{conj}\label{cj1} For any integer $n\geq 5$, there exists a constant $K(n)>0$ such that, for all smooth projective $n$-folds $X$ with $\vol(X)>K(n)$, $\varphi_{m,X}$ is birational for all $m\geq r_{n-1}$.
\end{conj}

\begin{conj}\label{cj2} For any integer $n\geq 6$, there exists a constant $L(n)$ such that, for all smooth projective $n$-folds $X$ with $p_g(X)>L(n)$, $\varphi_{m,X}$ is birational for all $m\geq r_{n-1}^+$. \end{conj}

It is also very interesting to compare two key numbers in the statement of Theorem \ref{main}. We have the following question which has an affirmative answer for $n=3,4$:

\begin{qu}\label{cj3} For any $n\geq 5$, is it true that
$$r_{n-1}\geq \rounddown{2\prod_{j=2}^{n-1}\Big(1+(n-j)\sqrt[n-j]{\frac{2}{v_{n-j}}}\ \Big)}?$$
\end{qu}

Clearly the positive answer to Question \ref{cj3} implies that Conjecture \ref{cj1} is true. Recently Gavin Brown kindly informed us of many interesting examples of canonical 4-folds. In fact,
in the  paper \cite{B-K}, one can find the following 4-fold:

\begin{exam}  Let $\bP=\bP(a_{10}, a_{12}, a_{17}, a_{33}, a_{37}, a_{55})$ be a weighted projective space, where the weight of the homogenous coordinate $a_i$ is $i$. Let $X\subset \bP$  be a general hypersurface of degree 165.
Then $X$ has canonical singularities and $K_X=\cO_X(1)$. Hence $X$ has the canonical volume $K_X^4=1/830280$. Since $X$ is general, the defining equation of $X$ can be written as $a_{55}^3+\{\mathrm{other\;\; terms}\}$. On the other hand, it is easy to check that the space  $H^0(X, \mathcal{O}_X(93))$ is generated by the variables $a_{10}, a_{12}, a_{17}, a_{33}, a_{37}$. Hence, for the degree-$165$ hypersurface $X$, $\varphi_{93}$ is not birational and thus $r_4\geq 94$.
\end{exam}

In the last part of this paper, we establish a general approach and show that
 a generalization of Theorem \ref{extension} would give an affirmative answer to Conjecture \ref{cj1}.

We begin with the following definition.
 \begin{defi}\label{ee} Given a birationally bounded set $\mathfrak{X}$ of smooth projective varieties and given a positive number $c$,
 we say that a fibraion $f: X\rightarrow T$ between smooth projective varieties satisfies condition $(B)_{ \mathfrak{X},   c}$ if
 \begin{itemize}
 \item[(1)] a general fiber $F$ of $f$ is birationally equivalent to an element of $ \mathfrak{X}$;
 \item[(2)] for a general point $t\in T$, there exists an effective $\bQ$-divisor $D_t$ with $D_t\sim_{\mathbb{Q}} \epsilon K_X$ for a positive rational number $\epsilon <c$, such that the fiber $F_t=f^{-1}(t)$ is an irreducible component of $\mathrm{Nklt}(X, D_t)$.
 \end{itemize}
 \end{defi}

The following question is an analogue of Theorem \ref{1.2} in higher dimensional base case.
 
\begin{qu}\label{qu-extension1} Given a birationally bounded set $\mathfrak{X}$ and a positive integer $t$, does there exist a number $c=c(\mathfrak{X}, t)$ such that, for any smooth projective variety $X$ of general type and a fibration $f: X\rightarrow T$ satisfying $(B)_{\mathfrak{X}, c}$ and $\dim T=t$,
 the restriction map $H^0(X, mK_X)\rightarrow H^0(F_1, mK_{F_1})\oplus H^0(F_2, mK_{F_2})$ is surjective for two arbitrary general fibers $F_1$ and $F_2$ of $f$ and for all $m\geq 2$.
\end{qu}

\begin{rema} (1)
When $\dim T=1$, a fibration $f: X\rightarrow T$ satisfies property $(B)_{\mathfrak{X}, c}$ if and only if  $F$ is birational to an element of $\mathfrak{X}$ and $\vol(X)>M_c$ for some positive integer $M_c$.

(2) However, when $\dim T\geq 2$, to work on the extension problem as Question \ref{qu-extension1}, we cannot simply assume that $\vol(X)\gg \vol(F)$. For instance, take $X=C\times V\times F$ and $T=C\times V$, where $C$ is a smooth projective curve with genus $g( C)>1$ and $V$ and $F$ are smooth projective varieties of general  type. Let $f: X\rightarrow T$ be the natural projection. For any $m>0$, we can choose certain $V$ and $F$ such that $p_m(V)=0$ and $p_m(F)>0$. Then, when $g( C)$ is large enough, $\vol(X)\gg\vol(F)$, however the restriction map $H^0(X, mK_X)\rightarrow H^0(F, mK_F)$ is never surjective. Hence property (2) in Definition \ref{ee} is reasonable.
\end{rema}

We now apply Takayama and McKernan's induction process once again to obtain the following unconditional result, which has its own flavor.

\begin{theo}\label{bicanonical} Fix a function $\lambda: \mathbb{Z}_{>0}\times\mathbb{Z}_{>0}\rightarrow \mathbb{R}_{>0}$.
There exists integers $M_{n-1}>M_{n-2}>\ldots>M_1>0$
and $K>0$  such that for any smooth projective $n$-fold $X$ with $\vol(X)\geq K$,
all pluricanonical map $\varphi_a$ of $X$ is birational for $a\geq 2$, unless that,  after birational modifications, $X$ admits  a fibration $f: X\rightarrow Z$
which satisfies $(B)_{\mathfrak{X}_{k, M_{k}^k}, \lambda(k, M_k^k)}$ for some $1\leq k\leq n-1$.
\end{theo}
\begin{rema} Note that bicanonical map of surfaces has been intensively studied. In particular, Reider \cite{Re} proved that the bicanonical map of a smooth projective surface $S$
is always birational if $\vol(S)\geq 10$, unless $S$ presents the standard case, namely after birational modification of $S$,
there exists a fibration $f: S\rightarrow B$ such that a general fiber of $f$ is a genus $2$ curve.
We can also work out the explicit (not optimal) bounds for $3$-folds and, once again, it seems rather hard to have an explicit version 
of the above theorem in dimension $\geq 4$.
\end{rema}

\begin{proof}

For each $1\leq d\leq n-1$, we define the positive numbers $M_d$ in the following way.  First, take $M_1$ to be a positive number such that
 $$2^{n-1}(\prod_{i=1}^{n-1}(1+\frac{i\sqrt[i]{2}}{M_1})-1)<1.$$   Inductively,
if $M_{d-1}$ is defined, take $M_d>M_{d-1}$  such that \begin{eqnarray}\label{estimation2}
 2^{n-d}(\prod_{i=d}^{n-1}(1+\frac{i\sqrt[i]{2}}{M_d})-1)<\lambda(d-1, M_{d-1}^{d-1}) .
 \end{eqnarray}

 We then run the effective induction of Takayama and McKernan again to finish the proof of Theorem \ref{bicanonical}.

 To initiate the induction, pick two very general points $x_1$ and $x_2$ of $X$. We first take $D_{1}\sim_{\mathbb{Q}}t_{1}K_X$   with $t_{1}\leq n2^{\frac{1}{n}}\vol(K_X)^{-\frac{1}{n}}+\epsilon$ such that $x_1$, $x_2\in \mathrm{Nklt}(X, D_{1})$ and that $(X, D_1)$ is log canonical at $x_1$ with the minimal lc center $V_1$ containing $x_1$.

The induction goes as follows. Assume $x_1$, $x_2\in  \mathrm{Nklt}(X, D_d)$, $ (X, D_d)$ is lc at $x_1$ with $V_d$ the minimal lc centre of $(X, D_d)$ containing $x_1$, $\dim V_d=n_d\leq n-d$, and $D_d\sim_{\mathrm{Q}}t_dK_X$. Then,
\begin{itemize}
\item[$(1)_d$] either $\vol( V_d)\geq M_{n_d}^{n_d}$ and by Takayama's induction, there exists $D_{d+1}\sim_{\mathrm{Q}}t_{d+1}K_X$ with

\begin{eqnarray}\label{esti3} t_{d+1}<(1+\frac{n_d\sqrt[n_d]{2}}{M_{n_d}})t_d+2\frac{n_d\sqrt[n_d]{2}}{M_{n_d}}+\epsilon\end{eqnarray}
such that  $x_1$, $x_2\in  \mathrm{Nklt}(X, D_{d+1})$, $ (X, D_{d+1})$ is lc at $x_1$ with $V_{d+1}$ the minimal lc centre of $(X, D_d)$ containing $x_1$, and $\dim V_{d+1}=n_{d+1}\leq n-d-1$;
\item[(2)] or $\vol(V_d)<M_{n_d}^{n_d}
 $, and by McKernan's lemma (see \cite[Lemma 3.2]{Tod}), we have either
\begin{itemize}
\item[$(2.1)_d$] $\deg \pi \geq 2$, then there exists $D_{d+1}\sim_{\mathrm{Q}}t_{d+1}K_X$ with
\begin{eqnarray}\label{esti4}
t_{d+1}<2t_d+\epsilon,
  \end{eqnarray}
such that  $x_1$, $x_2\in  \mathrm{Nklt}(X, D_{d+1})$, $ (X, D_{d+1})$ is lc at $x_1$ with $V_{d+1}$  the minimal lc centre of $(X, D_d)$ containing $x_1$, and $\dim V_{d+1}=n_{d+1}\leq n-d-1$; or
\item[$(2.2)_d$] after birational modifications of $X$, we have a fibration $f_{n_d}: X\rightarrow Z$ with a general fiber $F\in \mathfrak{X}_{n_d, M_{n_d}^{n_d}}$ and $F$ is a pure log canonical centre for $(X, D_F)$ with $D_F\sim_{\mathbb{Q}} t_dK_X$.
\end{itemize}
\end{itemize}
Combining $(\ref{esti3})$ and $(\ref{esti4})$, we see that we always have
$$t_{k+1}\leq 2(1+\frac{n_k\sqrt[n_k]{2}}{M_{n_k}})t_k+2^k\frac{n_k\sqrt[n_k]{2}}{M_{n_k}}+\epsilon,$$
for $k\geq 2$.

The induction   stops when arriving at $(2.2)_d$ for some $d$, or we run the induction to the end and get $D_m\sim_{\mathbb{Q}}t_mK_X$ such that $\dim V_m=0$.

In the first case,  we have
\begin{eqnarray*}
t_d< n\frac{\sqrt[n]{2}}{\sqrt[n]{\vol(X)}}2^{d-1}\prod_{i=1}^{d-1}(1+\frac{n_i\sqrt[n_i]{2}}{M_{n_i}})+2^{d-1}(\prod_{i=1}^{d-1}(1+\frac{n_i\sqrt[n_i]{2}}{M_{n_i}})-1)+\epsilon.
\end{eqnarray*}

By (\ref{estimation2}) and the fact that $M_i<M_{i+1}$ and $n_i\leq n-i$, we see that when $\vol(X)$ is large enough,
\begin{eqnarray*}
t_d &<& 2^{d-1}(\prod_{i=1}^{d-1}(1+\frac{n_i\sqrt[n_i]{2}}{M_{n_{d-1}}})-1)\\
&\leq & 2^{n-n_d-1}(\prod_{i=1}^{d-1}(1+\frac{n_i\sqrt[n_i]{2}}{M_{n_{d-1}}})-1)\\
&\leq & 2^{n-n_d-1}(\prod_{k=n_d+1}^{n-1}(1+\frac{k\sqrt[k]{2}}{M_{n_d+1}})-1)\\
& < & \lambda(n_d, M_{n_d}^{n_d}).
\end{eqnarray*}
 Thus, $f_{n_d}$ satisfies $(B)_{\mathfrak{X}_{n_d, M_{n_d}^{n_d}}, \lambda(n_d, M_{n_d}^{n_d})}$.

In the second case, we have
\begin{eqnarray*}
t_n< n\frac{\sqrt[n]{2}}{\sqrt[n]{\vol(X)}}2^{n-1}\prod_{i=1}^{n-1}(1+\frac{n_i\sqrt[n_i]{2}}{M_{n_i}})+2^{n-1}(\prod_{i=1}^{n-1}(1+\frac{n_i\sqrt[n_i]{2}}{M_{n_i}})-1)+\epsilon.
\end{eqnarray*}
Once again, when $\vol(X)$ is large enough,
$$t_n<2^{n-1}(\prod_{i=1}^{n-1}(1+\frac{i\sqrt[i]{2}}{M_1})-1)<1.$$
Thus, by the point separation principle, we conclude that   $\varphi_a$ is birational for all $a\geq 2$.
\end{proof}

\begin{prop}
 If Question \ref{qu-extension1} has an affirmative answer, then Conjecture \ref{cj1} holds.
\end{prop}

\begin{proof}
Assume that Question \ref{qu-extension1} has an affirmative answer, we then define the function
 \begin{eqnarray*}
 && \lambda : \mathbb{Z}_{>0}\times\mathbb{Z}_{>0}\rightarrow \mathbb{R}_{>0}\\&& \lambda(a, b)=c(\mathfrak{X}_{a, b}, n-a).
\end{eqnarray*}
By Theorem \ref{bicanonical}, there exists integers $M_{n-1}>M_{n-2}>\ldots>M_1>0$
and $K>0$  such that for any smooth projective $n$-fold $X$ with $\vol(X)\geq K$,
the $\varphi_m$ is birational for $m\geq 2$, unless that,  after birational modifications, $X$ admits a fibration $f: X\rightarrow Z$
which satisfies $(B)_{\mathfrak{X}_{k, M_{k}^k}, \lambda(k, M_k^k)}$ for some $1\leq k\leq n-1$.
But in the latter case, by the definition of $\lambda$ and  the assumption that Question \ref{qu-extension1} has an affirmative answer, 
 the map $\varphi_{m}$ is birational, for all $m\geq r_k$.

Combining all these together, we see that $\varphi_a$ is birational for all $a\geq r_{n-1}$.
\end{proof}

\noindent{\bf Acknowledgement.} Meng Chen was supported by National Natural Science Foundation of China (\#11231003, \#11421061, \#11571076) and Program of Shanghai Academic Researcher Leader (Grant no. 16XD1400400). Zhi Jiang is supported by China's Recruitment Program of Global Experts. We are grateful to Gavin Brown who is so kind to send us his paper \cite{B-K} and to inform of their interesting results. Thanks are also due to Chenyang Xu who helped us a lot in understanding 
the new proof of Theorem \ref{kaw-extension} using \cite{BCHM}.  Meng Chen thanks Weiping Li for the hospitality during his visit at HKUST in the spring of 2015. Zhi Jiang thanks Fudan University for the support during his visits in Shanghai in the summer of 2015.

\end{document}